\documentclass[a4paper,11pt]{amsart}

\usepackage[left=2.7cm,right=2.7cm,top=3.5cm,bottom=3cm]{geometry}

\newfont{\cyr}{wncyr10 scaled 1100}

\usepackage{amsthm,amssymb,amsmath,amsfonts,mathrsfs,amscd}
\usepackage[latin1]{inputenc}
\usepackage[all]{xy}

\theoremstyle{definition}
\newtheorem{defi}{Definition}[section]
\newtheorem{ass}[defi]{Assumption}

\theoremstyle{remark}
\newtheorem{rem}[defi]{Remark}

\theoremstyle{plain}
\newtheorem{lemma}[defi]{Lemma}
\newtheorem{coro}[defi]{Corollary}
\newtheorem{teo}[defi]{Theorem}
\newtheorem{prop}[defi]{Proposition}

\newcommand{\Spec}{\mathrm{Spec}}
\newcommand{\Pic}{\mathrm{Pic}}
\newcommand{\End}{\mathrm{End}}
\newcommand{\Aut}{\mathrm{Aut}}
\newcommand{\Frob}{\mathrm{Frob}}

\newcommand{\res}{\mathrm{res}}
\newcommand{\Ker}{\mathrm{Ker}}
\newcommand{\Hom}{\mathrm{Hom}}

\newcommand{\Div}{\mathrm{Div}}

\newcommand{\Gal}{\mathrm{Gal}}
\newcommand{\GL}{\mathrm{GL}}
\newcommand{\PGL}{\mathrm{PGL}}

\newcommand{\Sel}{\mathrm{Sel}}

\newcommand{\Q}{\mathbb Q}
\newcommand{\N}{\mathbb N}
\newcommand{\Z}{\mathbb Z}
\newcommand{\R}{\mathbb R}
\newcommand{\C}{\mathbb C}
\newcommand{\F}{\mathbb F}
\newcommand{\T}{\mathbb T}
\newcommand{\W}{\mathcal W}

\newcommand{\Ta}{\mathrm{Ta}}
\newcommand{\fr}{\mathfrak}
\newcommand{\cl }{\mathcal}
\newcommand{\M}{\mathrm{M}}

\newcommand{\new}{{\rm new}}
\newcommand{\old}{{\rm old}}

\newcommand{\longmono}{\mbox{$\lhook\joinrel\longrightarrow$}}

\newcommand{\longepi}{\mbox{$\relbar\joinrel\twoheadrightarrow$}}

\newcommand{\dirlim}{\mathop{\varinjlim}\limits}

\newcommand{\Sha}{\mbox{\cyr{X}}}

\begin{document}

\title[On the vanishing of Selmer groups]{On the vanishing of Selmer groups for elliptic curves over ring class fields}

\author{Matteo Longo and Stefano Vigni}
\address{Dipartimento di Matematica, Universit\`a di Milano, Via C. Saldini 50, 20133 Milano, Italy}
\email{matteo.longo1@unimi.it}
\email{stevigni@mat.unimi.it}
\subjclass[2000]{11G05, 11G40}
\keywords{elliptic curves, Selmer groups, Birch and Swinnerton-Dyer conjecture}

\begin{abstract}
Let $E_{/\Q}$ be an elliptic curve of conductor $N$ without complex multiplication and let $K$ be an imaginary quadratic field of discriminant $D$ prime to $N$. Assume that the number of primes dividing $N$ and inert in $K$ is odd, and let $H_c$ be the ring class field of $K$ of conductor $c$ prime to $ND$ with Galois group $G_c$ over $K$. Fix a complex character $\chi$ of $G_c$. Our main result is that if $L_K(E,\chi,1)\neq0$ then $\Sel_p(E/H_c)\otimes_\chi\W=0$ for all but finitely many primes $p$, where $\Sel_p(E/H_c)$ is the $p$-Selmer group of $E$ over $H_c$ and $\W$ is a suitable finite extension of $\Z_p$ containing the values of $\chi$. Our work extends results of Bertolini and Darmon to almost all non-ordinary primes $p$ and also offers alternative proofs of a $\chi$-twisted version of the Birch and Swinnerton-Dyer conjecture for $E$ over $H_c$ (Bertolini and Darmon) and of the vanishing of $\Sel_p(E/K)$ for almost all $p$ (Kolyvagin) in the case of analytic rank zero.  
\end{abstract}

\maketitle


\section{Introduction}

Let $E_{/\Q}$ be an elliptic curve of conductor $N$ without complex multiplication and denote by $f(q)=\sum_{n=1}^\infty a_nq^n$ the normalized newform of weight $2$ on $\Gamma_0(N)$ associated with $E$ by the Shimura--Taniyama correspondence.

Let $K$ be an imaginary quadratic field of discriminant $D$ prime to $N$. The extension $K/\Q$ determines a factorization
\[ N=N^+N^- \]
where a prime number $q$ divides $N^+$ (respectively, $N^-$) if and only if it is split (respectively, inert) in $K$. We make the following

\begin{ass} \label{ass1}
$N^-$ is square-free and the number of primes dividing it is \emph{odd}.
\end{ass}
Let $c$ be a positive integer prime to $ND$ and denote $\cl O_c$ the order of $K$ of conductor $c$: if $\cl O_K$ is the ring of integers of $K$ then $\cl O_c=\Z+c\cl O_K$. Let $H_c=K\bigl(j(\cl O_c)\bigr)$ be the ring class field of $K$ of conductor $c$; here $j$ is the classical $j$-function viewed as a function on lattices. The Galois group $G_c:=\Gal(H_c/K)$ of $H_c$ over $K$ is canonically isomorphic to the Picard group $\Pic(\cl O_c)=\widehat{\cl O}_c^\times\backslash\widehat K^\times/K^\times$ of $\cl O_c$ via class field theory. (For any ring $A$, the symbol $\widehat A$ denotes $A\otimes_\Z\widehat\Z$ where $\widehat\Z$ is the profinite completion of $\Z$.)

Write $\widehat G_c$ for the group of complex-valued characters of $G_c$, fix $\chi\in\widehat G_c$ and denote $L_K(f,\chi,s)$ the twist by $\chi$ of the $L$-function $L_K(f,s)=L_K(E,s)$ of $f$ over $K$. Since $c$ is prime to $ND$, it follows from Assumption \ref{ass1} that the sign of the functional equation of $L_K(f,\chi,s)$ is $+1$.

Now let $\Z[\chi]$ be the cyclotomic subring of $\C$ generated over $\Z$ by the values of $\chi$. For any prime number $p$ choose a prime ideal $\fr p$ of $\Z[\chi]$ containing $p$, and denote $\W$ the completion of $\Z[\chi]$ at $\fr p$. If $M$ is a $\Z[G_c]$-module, write $M\otimes_\chi\W$ for the tensor product of $M$ and $\W$ over $\Z[G_c]$, where the $\Z[G_c]$-module structure of $\W$ is induced by $\chi$.

Let $\Sel_p(E/H_c)$ and $\Sha_p(E/H_c)$ be the $p$-Selmer and the $p$-Shafarevich--Tate group of $E$ over $H_c$, respectively. The main result of our paper is the following

\begin{teo} \label{thm-intro-2}
If $L_K(f,\chi,1)\neq 0$ then
\[ \Sel_p(E/H_c)\otimes_\chi\W=0,\qquad \Sha_p(E/H_c)\otimes_\chi\W=0 \]
for all but finitely many primes $p$ and a suitable choice of $\fr p$.
\end{teo}

By purely algebraic considerations, we also show that Theorem \ref{thm-intro-2} yields the corresponding statement with $p$
replaced by $p^n$ for all $n\geq1$. The primes $p$ for which Theorem \ref{thm-intro-2} does not possibly hold are those not satisfying Assumption \ref{ass}. In particular, the set of such primes contains the primes of bad reduction for $E$ and those dividing the algebraic part $\cl L(f,\chi)$ of the special value $L_K(f,\chi,1)$, which is defined in Section \ref{section-GZ-formula} and is not zero because $L_K(f,\chi,1)$ is assumed to be non-zero. Furthermore, for primes $p$ for which Theorem \ref{thm-intro-2} holds the ideal $\fr p$ is chosen at the beginning of Section \ref{algebraic-sec} in such a way that $\cl L(f,\chi)$ is non-zero in $\Z[\chi]/\fr p$. 

Analogous results were previously obtained by Bertolini and Darmon
\begin{itemize}
\item for the finitely many primes $p$ of multiplicative reduction for $E$ which are inert in $K$ (in \cite[Theorem B]{BD-annals-146});
\item for infinitely many primes $p$ of \emph{ordinary} reduction for $E$ and $\chi$ of $p$-power conductor (in \cite[Corollary 4]{BD}).
\end{itemize}
It is important to remark that our Theorem \ref{thm-intro-2} does not \emph{a priori} exclude the case where $p$ is a prime of good \emph{supersingular} reduction for $E$, thus covering infinitely many $p$ not considered in \cite[Theorem B]{BD-annals-146} or \cite[Corollary 4]{BD}. The simple, yet crucial, observation which allows us to treat these cases as well is the following: if $F$ is a finite extension of $\Q_p$ and $A$ is an abelian variety defined over $F$ with good reduction then the image of the local Kummer map
\[ \delta:A(F)/pA(F)\;\longmono\;H^1(F,A[p]) \]
can be controlled by means of suitable flat cohomology groups (see \S \ref{flat-subsec}). In an Iwasawa-theoretic context, a similar approach was also adopted by Knospe in \cite{Kno}. The reader may wish to consult the paper \cite{DI} by Darmon and Iovita for related results on Iwasawa's Main Conjecture for elliptic curves in the supersingular case.

Now we would like to describe two interesting consequences of Theorem \ref{thm-intro-2}. First of all, the group $G_c$ acts naturally on the Mordell--Weil group $E(H_c)$, and $E(H_c)\otimes_\Z\C$ can be decomposed into a direct sum of eigenspaces under the induced action. Explicitly,
\[ E(H_c)\otimes\C=\bigoplus_{\chi\in\widehat G_c}E(H_c)^\chi \]
where $\widehat G_c$ is the group of complex-valued characters of $G_c$ and
\[ E(H_c)^\chi:=\bigl\{x\in E(H_c)\otimes\C\mid\text{$\sigma(x)=\chi(\sigma)x$ for all $\sigma\in G_c$}\bigr\}. \]
As explained in Sections \ref{algebraic-sec} and \ref{selmer-sec}, as a corollary of Theorem \ref{thm-intro-2} we get the following result on the vanishing of $E(H_c)^\chi$.
\begin{teo}[Bertolini--Darmon] \label{thm-intro}
If $L_K(f,\chi,1)\neq 0$ then $E(H_c)^\chi=0$.
\end{teo}
This is the $\chi$-twisted conjecture of Birch and Swinnerton-Dyer for $E$ over $H_c$ in the case of analytic rank zero and was previously established by Bertolini and Darmon in \cite[Theorem B]{BD-annals-146}. We remark that if $p$ is a prime of ordinary reduction for $E$ satisfying arithmetic conditions analogous to those in Assumption \ref{ass} and $\chi$ is an anticyclotomic character of $p$-power conductor then the statement of Theorem \ref{thm-intro} can also be found in \cite[Corollary 4]{BD}.

The second by-product of Theorem \ref{thm-intro-2} we want to mention is the following. By specializing Theorem \ref{thm-intro-2} to the trivial character of $G_c$, one can obtain a vanishing result for the $p$-Selmer groups of $E$ \emph{over} $K$. More precisely, we offer an alternative proof of
\begin{teo}[Kolyvagin] \label{sel-K-intro-teo}
If $L_K(E,1)\not=0$ then the Mordell--Weil group $E(K)$ is finite and $\Sel_p(E/K)=\Sha_p(E/K)=0$ for all but finitely many primes $p$.
\end{teo}
We attribute this result to Kolyvagin because it is a formal consequence of his proof of the conjectures of Birch and Swinnerton-Dyer and of Shafarevich and Tate for $E_{/K}$ (see Theorem \ref{kol-teo} for details), which were obtained by using his theory of Euler systems of Heegner points in rank one. However, our proof of Theorem \ref{sel-K-intro-teo} is new, since it uses neither known cases of the conjecture of Shafarevich and Tate nor auxiliary results for elliptic curves in rank one (in other words, it takes place ``entirely in rank zero'').\\ 

The methods used in the proof of Theorem \ref{thm-intro-2} are inspired by those of \cite{BD}. However, no techniques in Iwasawa theory are used in the course of our arguments, while a crucial role is played instead by a detailed study of the linear algebra of Galois cohomology groups of $E$ viewed as modules over $\Z_p[G_c]$ and $\F_p[G_c]$ (Sections \ref{algebraic-sec} and \ref{selmer-sec}). 

As hinted at in the lines above, the main new ingredient in our strategy is the use of flat cohomology to describe the image of the local Kummer maps above $p$ (Proposition \ref{prop-flat}). This approach can be effectively combined with classical results of Raynaud on $p$-torsion group schemes in order to control the local behaviour at primes of good reduction for $E$ of certain cohomology classes coming from points on the jacobians of suitable auxiliary Shimura curves (Proposition \ref{prop-kummer}). In particular, we do not need to require any ordinariness condition at $p$, contrary to what done, e.g., in \cite{BD}. However, we warn the reader that all this works fine under an assumption of ``low'' ramification in $p$ which is certainly satisfied in our case once we ask that $p\nmid c$ but fails to hold in other significant arithmetic contexts (for instance, when one deals with $\Z_p$-extensions of number fields). 

We expect that our approach to vanishing results for Selmer groups can be extended and fruitfully applied also to the context of real quadratic fields and Stark--Heegner points, as introduced by Darmon in \cite{da}: we plan to turn to this circle of ideas in a future project.

To conclude this introduction, we point out that the $\chi$-twisted, rank one situation was dealt with by Bertolini and Darmon in \cite{BD-RCF}. More precisely, building on the techniques of Kolyvagin (see, e.g., \cite{G2}, \cite{ko}, \cite{ru}), Bertolini and Darmon showed that the $\chi$-eigenspace $E(H_c)^\chi$ of the Mordell--Weil group $E(H_c)$ is one-dimensional over $\C$ if the projection onto $E(H_c)^\chi$ of a certain Heegner point is non-zero. Recently, this result has been largely extended by Nekov\'a\v{r} in \cite{Ne}, where the author covers the more general case of abelian varieties which are simple quotients of jacobians of Shimura curves associated to indefinite quaternion algebras over totally real number fields. We also remark that the main result of \cite{Ne} is one of the ingredients in the proof by Nekov\'a\v{r} of the parity conjecture for Selmer groups of Hilbert modular forms (see \cite[Ch. 12]{Ne-selmer}). 

\vskip 2mm

\noindent\emph{Notation and conventions.} Throughout our work we fix an algebraic closure $\bar\Q$ of $\Q$ and view all number fields as subfields of $\bar\Q$. If $F$ is a number field we write $\cl O_F$ and $G_F$ for the ring of integers and the absolute Galois group $\Gal(\bar\Q/F)$ of $F$, respectively. Moreover, for all primes $\ell$ we choose an algebraic closure $\bar\Q_\ell$ of $\Q_\ell$ and denote $\C_\ell$ its completion.

If $\ell$ is a prime then $\F_\ell$ and $\F_{\ell^2}$ are the finite fields with $\ell$ and $\ell^2$ elements, respectively. We often write $\F_p$ in place of $\Z/p\Z$ when we want to emphasize the field structure of $\Z/p\Z$.

For any ring $R$ and any pair of maps $f:M\rightarrow N$, $g:P\rightarrow Q$ of $R$-modules we write $f\otimes g:M\otimes_RP\rightarrow N\otimes_RQ$ for the $R$-linear map obtained by extending additively the rule $m\otimes p\mapsto f(m)\otimes g(p)$. Finally, for any map $f:M\rightarrow N$ of $R$-modules and any $R$-algebra $S$ the map $f\otimes 1:M\rightarrow N\otimes_RS$ is defined by $m\mapsto f(m)\otimes 1$.

\vskip 2mm

\noindent\emph{Acknowledgements.} We would like to thank Massimo Bertolini and Adrian Iovita for useful discussions and comments on some of the topics of this paper. We are also grateful to Marco Seveso for a careful reading of an earlier version of the article and to the anonymous referee for valuable remarks and suggestions.  

\section{Consequences of the Gross--Zhang formula} \label{section-GZ-formula}

Let $B$ be the definite quaternion algebra over $\Q$ of discriminant $N^-$, and let $R\subset B$ be an Eichler order of level $N^+$. Fix an optimal embedding $\psi:K\hookrightarrow B$ of $\cl O_c$ into $R$, that is, an injective $\Q$-algebra homomorphism of $K$ into $B$ such that $\psi(\cl O_c)=\psi(K)\cap R$. Extend $\psi$ to a homomorphism
\[ \widehat\psi:\widehat{\cl O}_c^\times\backslash\widehat K^\times/K^\times\longrightarrow\widehat R^\times\backslash\widehat B^\times/B^\times \]
in the obvious way. By Jacquet--Langlands theory, the modular form $f$ can be associated to a function
\[ \phi=\phi(f):\widehat R^\times\backslash\widehat B^\times/B^\times\longrightarrow\Z \]
with the same eigenvalues of $f$ under the action of Hecke operators $T_q$ for primes $q\nmid N$, where the action of Hecke operators on $\widehat B^\times$ is defined by double coset decomposition. Define the \emph{algebraic part} of the special value $L_K(f,\chi,1)$ of $L_K(f,\chi,s)$ to be
\[ \cl L(f,\chi):=\sum_{\sigma\in G_c}\chi^{-1}(\sigma)\bigl(\phi\circ\widehat\psi\,\bigr)(\sigma)\in\Z[\chi].\]
Of course, $\cl L(f,\chi)$ depends on the choice of $\psi$ but, since this embedding will remain fixed throughout our work, the notation will not explicitly reflect this dependence.
\begin{teo} \label{Gross-formula}
The special value $L_K(f,\chi,1)$ is non-zero if and only if $\cl L(f,\chi)$ is non-zero.
\end{teo}
\begin{proof} The result in this form follows from \cite{W} and \cite{GP}: see \cite[Theorem 6.4]{V} for details. Explicit formulas relating $L_K(f,\chi,1)$ and $\cl L(f,\chi)$ can be found in \cite{G} in a special case and in \cite{Z} in the greatest generality. \end{proof}

\section{Local cohomology and Selmer groups}

In this section we introduce Selmer groups. Since we will consider torsion modules of cardinality divisible by the residue
characteristic of (some of) the local fields, in order to study local conditions we need to use both Galois and flat cohomology.

\subsection{Classical Selmer groups} \label{classical-selmer-subsec}

Let $F$ be a number field. For any prime number $q$ denote $F_q$ the $q$-adic completion $F\otimes_\Q\Q_q$ of $F$, so that $F_q$ is isomorphic to the product $\prod_{\fr q\mid q}F_\fr q$ of the completions $F_\fr q$ of $F$ at the prime ideals $\fr q$ of $\cl O_F$ above $q$. Moreover, write $G_{F_\fr q}$ for a decomposition group of $G_F$ at $\fr q$ (this amounts to choosing an algebraic closure $\bar F_{\fr q}$ of $F_{\fr q}$ and an embedding $\bar\Q\hookrightarrow\bar F_{\fr q}$) and $I_{F_\fr q}$ for the inertia subgroup of $G_{F_\fr q}$, and set $G_{F_q}:=\prod_{\fr q|q}G_{F_\fr q}$ and $I_{F_q}:=\prod_{\fr q|q}I_{F_\fr q}$.

Let $A_{/F}$ be an abelian variety defined over a number field $F$. If $p$ is a prime number write $A[p^n]$ for the $p^n$-torsion subgroup of $A(\bar\Q)$, so $A[p^n]\simeq(\Z/p^n\Z)^{2d}$ where $d$ is the dimension of $A$. Let now $p$ and $q$ be (possibly equal) primes and let $n$ be a positive integer. As customary, set
\[ H^1(F,A[p^n]):=H^1(G_F,A[p^n]),\qquad H^1(F_\fr q,A[p^n]):=H^1(G_{F_\fr q},A[p^n]) \]
and
\[ H^1(F_q,A[p^n]):=\prod_{\fr q|q}H^1(F_\fr q,A[p^n]),\qquad H^1(I_{F_q},A[p^n]):=\prod_{\fr q|q}H^1(I_{F_\fr q},A[p^n]). \]
(Here $H^1(G,M)$ denotes the first (continuous) cohomology group of the profinite group $G$ with values in the $G$-module $M$.) Let
\[ \res_\fr q:H^1(F,A[p^n])\longrightarrow H^1(F_\fr q,A[p^n]) \]
be the restriction map at $\fr q$ and define $\res_q:=\prod_{\fr q\mid q}\res_\fr q$. Likewise, let
\[ \delta_\fr q:A(F_\fr q)\longrightarrow H^1(F_\fr q,A[p^n]) \]
denote the Kummer map and define $\delta_q:=\prod_{\fr q\mid q}\delta_{\fr q}$. As usual, the $p^n$-Selmer group of $A$ over $F$ is defined as
\[ \begin{split}
   \Sel_{p^n}(A/F)&:=\ker\Big(H^1(F,A[p^n])\longrightarrow\displaystyle{\prod_{\fr q}} H^1\bigl(F_{\fr q},A(\bar F_\fr q)\bigr)[p^n]\Big)\\
   &\;=\bigl\{s\in H^1(F,A[p^n])\mid\text{$\res_q(s)\in{\rm Im}(\delta_q)$ for all primes $q$}\bigr\}.
   \end{split} \]
Finally, say that a prime $q$ is of \emph{good reduction} for $A$ if for all prime ideals $\fr q$ of $\cl O_F$ above $q$ the base-changed abelian variety $A_\fr q:=A\times_FF_\fr q$ has good reduction.

\subsection{Finite and singular cohomology groups}

As in \cite[\S 2.2]{BD}, we introduce the following finite/singular structures on local cohomology groups. Let $p$ and $q$ be \emph{distinct} primes and suppose that $A$ has good reduction at $q$. Define the \emph{singular part} of $H^1(F_q,A[p])$ as
\[ H^1_{\rm sing}(F_q,A[p]):=H^1(I_{F_q},A[p])^{G_{F_q}/I_{F_q}}, \]
and define the \emph{finite part} $H^1_{\rm fin}(F_q,A[p])$ via the inflation-restriction exact sequence
\[ 0\longrightarrow H^1_{\rm fin}(F_q,A[p])\longrightarrow H^1(F_q,A[p])\longrightarrow H^1_{\rm sing}(F_q,A[p]).\]
The next proposition recalls a well-known cohomological result.

\begin{prop} \label{prop-fin-sing}
Let $p$, $q$ be distinct primes and suppose that $A$ has good reduction at $q$. Then ${\rm Im}(\delta_q)=H^1_{\rm fin}(F_q,A[p])$.
\end{prop}

\begin{proof} Let $\fr q$ be a prime of $F$ above $q$. Since $A$ has good reduction at $\fr q$ and $p\not=q$, there is an exact sequence
\[ A(F_{\fr q})\overset{\delta_{\fr q}}{\longrightarrow}H^1\bigl(G_{F_{\fr q}}/I_{F_{\fr q}},A[p]\bigr)\longrightarrow H^1\bigl(G_{F_{\fr q}}/I_{F_{\fr q}},A(\bar\Q_q)^{I_{F_{\fr q}}}\bigr)[p]. \]
But $H^1\bigl(G_{F_{\fr q}}/I_{F_{\fr q}},A(\bar\Q_q)^{I_{F_{\fr q}}}\bigr)=0$ by \cite[Ch. I, Proposition 3.8]{Mi2}, and our claim follows. \end{proof}

\subsection{Flat cohomology groups} \label{flat-subsec}

Fix a prime ideal $\fr p$ of $\cl O_F$ above $p$ and suppose that $A$ has good reduction at $\fr p$. Let $\cl A$ be the Néron model of $A\times_FF_\fr p$ over the ring of integers $\cl O_{F_\fr p}$ of $F_\fr p$ and denote by $\cl A[p]$ the $p$-torsion subgroup scheme of $\cl A$. View the group schemes $\cl A$ and $\cl A[p]$ as sheaves on the flat site of $\Spec(\cl O_{F_\fr p})$, and write $H^1_{fl}(\cl O_{F_\fr p},\cl A)$ and $H^1_{fl}(\cl O_{F_\fr p},\cl A[p])$ (respectively, $H^2_{\cdot}(\cl O_{F_\fr p},\cl A)$ and $H^2_{\cdot}(\cl O_{F_\fr p},\cl A[p])$) for the first flat cohomology group (respectively, the second flat cohomology group supported on the closed point of $\Spec(\cl O_{F_\fr p})$) of $\cl A$ and $\cl A[p]$ (see \cite[Ch. III]{Mi1} for the theory of sites and cohomologies on the flat site). These groups fit into the following commutative diagram with exact rows, where the injectivity of $i_\fr p$ is a consequence of \cite[Ch. III, Lemma 1.1 (a)]{Mi2}:
\begin{equation} \label{diagram-flat}
\xymatrix{0\ar[r]& H^1_{fl}(\cl O_{F_\fr p},\cl A[p])\ar[r]^-{i_\fr p}\ar[d] & H^1(F_\fr p,A[p])\ar[r]\ar[d] & H^2_{\cdot}(\cl O_{F_\fr p},\cl A[p])\ar[d]\\
& H^1_{fl}(\cl O_{F_\fr p},\cl A)\ar[r] & H^1\bigl(F_\fr p,A(\bar F_\fr p)\bigr)\ar[r] & H^2_{\cdot}(\cl O_{F_\fr p},\cl A).}
\end{equation}
Now recall the local Kummer map $\delta_\fr p:A(F_{\fr p})\rightarrow H^1(F_\fr p,A[p])$ at $\fr p$ and the exact sequence
\[ A(F_{\fr p})\overset{\delta_\fr p}{\longrightarrow}H^1(F_\fr p,A[p])\longrightarrow H^1\bigl(F_\fr p,A(\bar F_\fr p)\bigr). \]
The next result will play a crucial role in the proof of Proposition \ref{prop-kummer}.
\begin{prop} \label{prop-flat}
The map $i_\fr p$ induces an isomorphism ${\rm Im}(\delta_\fr p)\simeq H^1_{fl}(\cl O_{F_\fr p},\cl A[p])$.
\end{prop}

\begin{proof} Since $H^1_{fl}(\cl O_{F_\fr p},\cl A)=0$ by Lang's lemma (see \cite[Théorème 11.7]{Gro} and \cite[Lemma 5.1 (vi)]{Maz}), the inclusion
\[ i_\fr p\bigl(H^1_{fl}(\cl O_{F_\fr p},\cl A[p])\bigr)\subset{\rm Im}(\delta_\fr p) \]
follows from \eqref{diagram-flat}. Since $H^1_{\cdot}(\cl O_{F_\fr p},\cl A)=0$ by \cite[Lemma 5.1 (vi)]{Maz}, the right vertical map in \eqref{diagram-flat} is injective. Hence if $x$ belongs to the kernel of the middle vertical arrow then its image in $H^2_{\cdot}(\cl O_{F_\fr p},\cl A[p])$ is zero, and the opposite inclusion follows. \end{proof}

\section{Admissible primes relative to $f$ and $p$} \label{section-Galois-representations}

For any prime number $q$, fix an isomorphism $E[q]\simeq(\Z/q\Z)^2$ by choosing a basis of $E[q]$ over $\Z/q\Z$ and let
\[ \rho_{E,q}:G_\Q\longrightarrow\GL_2(\Z/q\Z) \]
be the representation of the absolute Galois group of $\Q$ acting on $E[q]$.

\subsection{Choice of $p$}

Throughout our work we fix a prime number $p$ fulfilling the next

\begin{ass} \label{ass}
Suppose $L_K(f,\chi,1)\neq 0$. Then
\begin{enumerate}
\item $p\geq5$ and $p$ does not divide $cNDh(c)$ where $h(c):=[H_c:K]$ is the cardinality of the group $\text{Pic}(\cl O_c)$;
\item the Galois representation $\rho_{E,p}$ is surjective;
\item the image of $\cl L(f,\chi)$ in the quotient $\Z[\chi]/p\Z[\chi]$ is not zero;
\item $p$ does not divide the minimal degree of a modular parametrization $X_0(N)\rightarrow E$;
\item if $\fr q$ is a prime of $H_c$ dividing $N$ and $H_{c,\fr q}$ is the completion of $H_c$ at $\fr q$ then $p$ does not divide the order of the torsion subgroup of $E(H_{c,\fr q})$.
\end{enumerate}
\end{ass}
By Theorem \ref{Gross-formula}, the non-vanishing of $L_K(f,\chi,s)$ at $s=1$ is necessary for part $3$ to hold, and this is the reason why it was assumed before enumerating the required properties of the prime $p$. On the other hand, in order to emphasize its role in our arithmetic context, the condition $L_K(f,\chi,1)\not=0$ will always explicitly appear in the statement of each of our results in Section \ref{selmer-sec}.

\begin{rem}
A well-known theorem of Lutz (\cite{Lutz}) says that there is an isomorphism
\[ E(H_{c,\fr q})\simeq\Z_q^{[H_{c,\fr q}:\Q_q]}\times T \]
with $T$ a finite group, hence the torsion of $E(H_{c,\fr q})$ is indeed finite. Under the condition $L_K(f,\chi,1)\neq 0$, Theorem \ref{Gross-formula} implies that the assumption on $\cl L(f,\chi)$ excludes only a finite number of primes $p$. But then, since $E$ does not have complex multiplication, the ``open image'' theorem of Serre (\cite{S}) ensures that Assumption \ref{ass} is verified for all but finitely many primes $p$.
\end{rem}
\begin{rem}
Condition $5$ in Assumption \ref{ass} is introduced in order to ``trivialize'' the image of the local Kummer map at primes of bad reduction for $E$. However, with little extra effort one could impose suitable conditions at these primes too, thus relaxing Assumption \ref{ass}. We preferred avoiding this in order not to burden the exposition with unnecessary (at least for our purposes) technicalities, but local conditions of this kind appear, e.g., in \cite{GPa}.
\end{rem}

\subsection{Admissible primes} \label{admissible-subsec}

Let $p$ be our chosen prime number and recall the normalized newform $f(q)=\sum_{i=1}^\infty a_iq^i$ of weight $2$ on $\Gamma_0(N)$ associated with $E$. Following \cite[\S 2]{BD}, we say that a prime number $\ell$ is \emph{admissible relative to $f$ and $p$} (or simply \emph{admissible}) if it satisfies the following conditions:
\begin{enumerate}
\item $\ell$ does not divide $Npc$;
\item $\ell$ is inert in $K$;
\item $p$ does not divide $\ell^2-1$;
\item $p$ divides $(\ell+1)^2-a_\ell^2$.
\end{enumerate}
For every admissible prime $\ell$ choose once and for all a prime $\lambda_0$ of $H_c$ above $\ell$ (we will never deal with more than one admissible prime at the same time, so ignoring the dependence of $\lambda_0$ on $\ell$ should cause no confusion). Since it is inert in $K$ and it does not divide $c$, the admissible prime $\ell$ splits completely in $H_c$, hence the primes of $H_c$ above $\ell$ correspond bijectively to the elements of $G_c$. The choice of $\lambda_0$ allows us to fix an explicit bijection between these two sets via the rule
\begin{equation} \label{bijection-eq}
\sigma\in G_c\longmapsto\sigma(\lambda_0).
\end{equation}
The inverse to this bijection will be denoted
\begin{equation} \label{bijection-inverse-eq}
\lambda\longmapsto\sigma_\lambda\in G_c,
\end{equation}
so that $\sigma_\lambda(\lambda_0)=\lambda$. Finally, an element $\sigma\in G_c$ acts on the group rings $\Z[G_c]$ and $\Z/p\Z[G_c]$ in the natural way by multiplication on group-like elements (that is, $\gamma\mapsto\sigma\gamma$ for all $\gamma\in G_c$).
\begin{lemma} \label{local-iso-lemma}
Let $\ell$ be an admissible  prime relative to $f$ and $p$. The local cohomology groups $H^1_{\rm fin}(H_{c,\ell},E[p])$ and $H^1_{\rm sing}(H_{c,\ell},E[p])$ are both isomorphic to $\Z/p\Z[G_c]$ as $\Z[G_c]$-modules.
\end{lemma}
\begin{proof} By \cite[Lemma 2.6]{BD}, the groups $H^1_{\rm fin}(K_{\ell},E[p])$ and $H^1_{\rm sing}(K_{\ell},E[p])$ are both isomorphic to $\Z/p\Z$. But $\ell$ splits completely in $H_c$, hence $H^1_{\rm fin}(H_{c,\ell},E[p])$ and $H^1_{\rm sing}(H_{c,\ell},E[p])$ are both isomorphic to $\Z/p\Z[G_c]$ as $\F_p$-vector spaces. Finally, the bijection described in \eqref{bijection-eq} establishes isomorphisms which are obviously $G_c$-equivariant, and we are done. \end{proof}

For $\star\in\{\text{fin, sing}\}$ we fix once and for all isomorphisms
\begin{equation} \label{fin-sing-iso-eq}
H^1_\star(K_\ell,E[p])\simeq\Z/p\Z
\end{equation}
which will often be viewed as identifications according to convenience.

The following proposition is a variant of \cite[Theorem 3.2]{BD} (in the proof we will make use of the algebraic results described in Appendix \ref{appendix}).

\begin{prop} \label{existence-admissible-primes}
Let $s$ be a non-zero element of $H^1(H_c,E[p])$. There are infinitely many admissible primes $\ell$ such that $\res_\ell(s)\neq0$.
\end{prop}

\begin{proof} Let $\Q(E[p])$ be the extension of $\Q$ fixed by the kernel of the representation ${\rho}_{E,p}$. Let $M$ be the composite of the extensions $\Q(E[p])$ and $H_c$, which are linearly disjoint over $\Q$; in fact, the discriminant of $H_c$ is prime to $pN$ and $\Q(E[p])$ is ramified only at primes dividing $pN$, hence $H_c\cap\Q(E[p])$ is unramified over $\Q$: we conclude that $H_c\cap\Q(E[p])=\Q$ by Minkowski's theorem. Since $\Gal(H_c/\Q)$ is the semidirect product of $G_c$ and $\Gal(K/\Q)$, with the non-trivial element $\tau$ of $\Gal(K/\Q)$ acting on the abelian normal subgroup $G_c$ by $\sigma\mapsto\tau\sigma\tau^{-1}=\sigma^{-1}$, it follows that
\[ \Gal(M/\Q)=\Gal(H_c/\Q)\times \Gal(\Q(E[p])/\Q)\subset(G_c\rtimes\{1,\tau\})\times\Aut(E[p]). \]
The elements in $\Gal(M/\Q)$ can then be identified with triples $(\sigma,\tau^j,T)$ where $\sigma\in G_c$, $j\in \left\{0,1\right\}$ and $T\in\Aut(E[p])$. For any $s\in H^1(H_c,E[p])$ denote
\[ \bar{s}\in H^1(M,E[p])^{\Gal(M/H_c)}=\Hom_{\Gal(M/H_c)}(\Gal(\bar\Q/M),E[p]) \]
the restriction of $s$. By the argument in the proof of \cite[Proposition 4.1]{BD-RCF}, the above restriction map is injective, so $\bar{s}\neq 0$. Let $M_s$ be the extension of $M$ cut out by $\ker(\bar{s})$. By Proposition \ref{app-prop-1}, $M_s/H_c$ is Galois. Since $E[p]$ is an irreducible $\Gal(\bar\Q/H_c)$-module because $\rho_{E,p}$ is surjective by condition $2$ in Assumption \ref{ass}, it follows that $\Gal(M_s/M)\simeq E[p]$.

Fix now $s\in H^1(H_c,E[p])$ such that $s\neq 0$. Without loss of generality, assume that $s$ belongs to an eigenspace for the complex conjugation $\tau$, so $\tau(s)=\delta s$ for $\delta\in\{1,-1\}$. For every $\sigma\in G_c$ write $s\mapsto s^\sigma$ for the natural action of $\sigma$ on $s$ and let $\tilde M_s$ denote the composite of the extensions $M_{s^ \sigma}$ for all $\sigma\in G_c$. By Proposition \ref{app-prop-3}, $\tilde M_s/\Q$ is Galois and, by Proposition \ref{app-prop-4}, $\Gal(\tilde M_s/\Q)$ identifies with the semidirect product of $\Gal(M/\Q)$ and a finite number of copies of $E[p]$ indexed by a suitable subset $S$ of elements in $G_c$ (here we use the fact that $s$ belongs to a specific eigenspace for $\tau$). Denote elements in $\Gal(\tilde M_s/\Q)$ by $\bigl((v_\rho)_{\rho\in S},\sigma,\tau^j,T\bigr)$ with $(v_\rho)_{\rho\in S}\in E[p]^{\#S}$ and $(\sigma,\tau^j,T)\in \Gal(M/\Q)$. Let now $\bigl((v_\rho)_{\rho\in S},1,\tau,T\bigr)\in\Gal(\tilde M_s/\Q)$ be such that
\begin{enumerate}
\item $T$ has eigenvalues $\delta$ and $\mu\in(\Z/p\Z)^\times$ such that $\mu\neq\pm1$ (such a $T$ exists because $\rho_{E,p}$ is surjective and $p\geq 5$);
\item for all $\rho\in S$, $v_\rho=v\neq 0$ and $v$ belongs to the $\delta$-eigenspace for $T$.
\end{enumerate}
For any number field $F$, any Galois extension $\Gal(F'/F)$ and any prime ideal $\fr q$ of $F$ which is unramified in $F'$, let ${\rm Frob}_{F'/F}(\fr q)$ denote a Frobenius element of $\Gal(F'/F)$ at $\fr q$. Let $\ell$ be a prime number such that
\begin{enumerate}
\item $\ell$ does not divide $Npc$;
\item $\ell$ is unramified in $\tilde M_s/\Q$;
\item $\Frob_{\tilde M_s/\Q}(\ell)=\bigl((v_\rho)_{\rho\in S},1,\tau,T\bigr)$.
\end{enumerate}
By the \v Cebotarev density theorem, there are infinitely many such primes.

The prime $\ell$ is admissible. Indeed, first note that $\Frob_{K/\Q}(\ell)=\tau$ implies that $\ell$ is inert in $K$. Moreover, the characteristic polynomial of $\Frob_{M/K}(\ell)$ acting on $E[p]$ is the reduction modulo ${p}$ of $X^2-a_\ell
X+\ell$, so $a_\ell\equiv\delta+\mu\pmod{p}$ and $\ell\equiv\delta\mu\pmod{p}$, whence $a_\ell\equiv\delta(\ell+1)\pmod{p}$; we conclude that $a_\ell\equiv\pm(\ell+1)\pmod{p}$ since $\delta\in\{1,-1\}$. Finally, $\ell\not\equiv\pm1\pmod{p}$ since $\mu\neq\pm1$.

To show that $\res_\ell(s)\neq 0$, let $\fr l$ be a prime of $M$ dividing $\ell$ and define $r$ to be the degree of the corresponding residue field extension. By Proposition \ref{app-prop-4} and equation \eqref{app-eq-2},
\[ {\rm Frob}_{\tilde M_s/M}(\fr l)=\bigl((v_\rho)_{\rho\in S},1,\tau,T\bigr)^r=\bigl((rv_\rho)_{\rho\in S},1,1,1\bigr). \]
The integer $r$ is even and prime to $p$ because $p$ does not divide the cardinality $h(c)$ of $G_c$ by condition $1$ in Assumption \ref{ass} and the order of $T$ is prime to $p$. Hence
\[ \bar{s}\bigl(\Frob_{\tilde M_s/M}(\fr l)\bigr)=\bar{s}(rv)=r\bar{s}(v)\neq 0 \]
and $\res_{\fr l}(\bar{s})\neq 0$. This shows that $\res_{\fr l}(s)\neq 0$, hence $\res_\ell(s)\neq 0$. \end{proof}

\section{Shimura curves and Hecke algebras}

\subsection{The Shimura curve $X^{(\ell)}$}\label{section-X-ell}

Let $\ell$ be an admissible prime relative to $f$ and $p$, let $\cl B$ be the indefinite quaternion algebra over $\Q$ of discriminant $N^-\ell$ and fix an Eichler order $\cl R\subset\cl B$ of level $N^+$. Let $\cl H_\infty$ be the complex upper half plane and denote $\cl R^\times_1$ the group of elements of norm $1$ in $\cl R$. Fix an embedding $i_\infty:\cl B\rightarrow\M_2(\R)$ and write $\Gamma_\infty$ for the image of $\cl R_1^\times$ in $\PGL_2(\R)$ obtained by composing $i_\infty$ with the canonical projection $\GL_2(\R)\rightarrow\PGL_2(\R)$. Let $\PGL_2(\R)$ act on $\cl H_\infty$ by M\"obius (i.e., fractional linear) transformations. The analytic quotient $\cl H_\infty/\Gamma_\infty$ has a natural structure of compact Riemann surface which, by \cite[Ch. 9]{Sh}, admits a model $X^{(\ell)}=X^{(\ell)}_\Q$ defined over $\Q$. The curve $X^{(\ell)}$ will be referred to as the \emph{rational Shimura curve} associated to $\cl B$ and $\cl R$.

Let $\Q_{\ell^2}$ (respectively, $\Z_{\ell^2}$) be the (unique, up to isomorphism) quadratic unramified extension of $\Q_\ell$ (respectively, $\Z_\ell$) and let $\mathbb F_{\ell^2}$ be the field with $\ell^2$ elements. Denote by $J^{(\ell)}$ the Jacobian variety of $X^{(\ell)}$ and by $J^{(\ell)}_{\Z_{\ell^2}}$ its Néron model over $\Z_{\ell^2}$. Let $J^{(\ell)}_{\F_{\ell^2}}$ be the special fiber of $J^{(\ell)}_{\Z_{\ell^2}}$ and denote by $J^{(\ell),0}_{\F_{\ell^2}}$ the connected component of the origin in $J^{(\ell)}_{\F_{\ell^2}}$. Finally, write
\[ \Phi_\ell:=J^{(\ell)}_{\F_{\ell^2}}\big/J^{(\ell),0}_{\F_{\ell^2}} \]
for the group of (geometric) connected components of $J^{(\ell)}_{\F_{\ell^2}}$.

\subsection{Hecke algebras and liftings of modular forms}

For any integer $M$, let $\T(M)$ be the Hecke algebra generated over $\Z$ by the Hecke operators $T_q$ for primes $q\nmid M$ and $U_q$ for primes $q|M$ acting on the space of cusp forms of level $2$ on $\Gamma_0(M)$. Denote $\T^\new$ (respectively, $\T_\ell^\new$) the quotient of $\T(N)$ (respectively, $\T(N\ell)$) acting faithfully on the cusp forms on $\Gamma_0(N)$ which are new at $N^-$ (respectively, $N^-\ell$). In the following, the Hecke operators in $\T(N)$ (respectively, $\T(N\ell)$) will be denoted $T_q$ (respectively, $t_q$) for primes $q\nmid N$ (respectively, $q\nmid N\ell$) and $U_q$ (respectively, $u_q$) for primes $q|N$ (respectively, $q|N\ell$). To ease the notation, use the same symbols for the images of these operators via the projections $\T(N)\twoheadrightarrow\T^\new$ and $\T(N\ell)\twoheadrightarrow\T_\ell^\new$. The newform $f$ gives rise to surjective homomorphisms
\[ f:\T(N)\longrightarrow\Z,\qquad f:\T^\new\longrightarrow\Z \]
which, by a notational abuse, will be denoted by the same symbol, the second one being the factorization of the first one through the projection onto the new quotient. These maps are defined by extending $\Z$-linearly the rule $f(T_q):=a_q$, $f(U_q):=a_q$; by composing $f$ with the projection onto $\Z/p\Z$ we further obtain surjections
\[ \bar f:\T(N)\longrightarrow\Z/p\Z,\qquad \bar f:\T^\new\longrightarrow\Z/p\Z. \]
Under the above assumptions on $p$ and $\ell$, by \cite[Theorem 5.15]{BD} there exists a surjective homomorphism
\begin{equation} \label{f-ell}
f_\ell:\T_\ell^\new\longrightarrow\Z/p\Z 
\end{equation} 
such that
\begin{enumerate}
\item $f_\ell(t_q)=\bar f(T_q)$ for all primes $q\nmid N\ell$;
\item $f_\ell(u_q)=\bar f(U_q)$ for all primes $q|N$;
\item $f_\ell(u_\ell)=\epsilon\pmod{p}$ where $p$ divides $\ell+1-\epsilon f(T_\ell)$.
\end{enumerate}
Let $\fr m_{f_\ell}\subset\T_\ell^\new $ be the kernel of $f_\ell$. The proofs of \cite[Theorems 5.15 and 5.17]{BD} do not use the condition that $p$ is an ordinary prime for $E$ (which was assumed at the beginning of \cite{BD}), hence by \cite[Theorem 5.15]{BD} there is a group isomorphism
\begin{equation} \label{Phi-ell}
\Phi_\ell/\fr m_{f_\ell}\simeq\Z/p\Z,
\end{equation}
and by \cite[Theorem 5.17]{BD} there is an isomorphism
\begin{equation} \label{rep-iso}
\Ta_p(J^{(\ell)})/\fr m_{f_\ell}\simeq E[p]
\end{equation}
of $G_\Q$-modules.

\begin{rem}
The modular form $f_\ell$ introduced in \eqref{f-ell} and the isomorphism \eqref{Phi-ell} are obtained by Ribet's ``level raising'' argument (\cite[Theorem 7.3]{R}) combined with the fact that, by condition $4$ in Assumption \ref{ass}, the prime $p$ does not divide the degree of a minimal parametrization $X_0(N)\rightarrow E$. The argument is based on the isomorphism which will be briefly touched upon in \eqref{omega} (see \S \ref{hecke-operators-subsec} for more information), and the details can be found in \cite[Theorem 5.15]{BD}. 
\end{rem}

Let $F$ be a number field and let
\[ \kappa:J^{(\ell)}(F)\rightarrow H^1\bigl(F,{\rm Ta}_p(J^{(\ell)})\bigr) \]
be the Kummer map relative to $J^{(\ell)}$. Composing $\kappa$ with the canonical projection ${\rm Ta}_p(J^{(\ell)})\rightarrow {\rm Ta}_p(J^{(\ell)})/\fr m_{f_\ell}$ and the isomorphism \eqref{rep-iso} yields a map
\begin{equation} \label{kappa}
\bar\kappa:J^{(\ell)}(F)\longrightarrow H^1(F,E[p]).
\end{equation}
If $q$ is a prime let $\res_q:H^1(F,E[p])\rightarrow H^1(F_q,E[p])$ be the restriction map and let $\delta_q:E(F_q)\rightarrow H^1(F_q,E[p])$ (respectively, $\kappa_q:J^{(\ell)}(F_q)\rightarrow H^1(F_q,J^{(\ell)}[p])$) be the local Kummer map relative to $E$ (respectively, $J^{(\ell)}$). Finally, for any prime $\fr p|p$ denote $\cl E_\fr p$ (respectively, $\cl J^{(\ell)}_\fr p$) the Néron model of $E$ (respectively, $J^{(\ell)}$) over $\cl O_{F_\fr p}$. Moreover, let $v_\fr p$ be the (normalized) valuation of $F_\fr p$ and let $e:=v_\fr p(p)$ be the absolute ramification index of $F_\fr p$ (in particular, $e=1$ if $p$ is unramified in $F$).

The proof of the next proposition is where the result in flat cohomology shown in \S \ref{flat-subsec} comes into play.

\begin{prop} \label{prop-kummer}
Assume that $e<p-1$. If $P\in J^{(\ell)}(F)$ then
\[ \res_q\bigl(\bar\kappa(P)\bigr)\in{\rm Im}(\delta_q) \]
for all primes $q\nmid N\ell$.
\end{prop}

\begin{proof} If $q\not=p$ the claim follows from Proposition \ref{prop-fin-sing}. Suppose that $q=p$, note that $E$ and $J^{(\ell)}$ have good reduction at $p$ since $p\nmid N$ by condition $1$ in Assumption \ref{ass}, and write
\[ i_\fr p:H^1_{fl}(\cl O_{F_\fr p},\cl E_\fr p[p])\;\longmono\;H^1(F_\fr p,E[p]) \]
and
\[ j_\fr p:H^1_{fl}\bigl(\cl O_{F_\fr p},\cl J^{(\ell)}_\fr p[p]\bigr)\;\longmono\;H^1\bigl(F_\fr p,J^{(\ell)}[p]\bigr) \]
for the two injections in \eqref{diagram-flat} corresponding to $A=E$ and $A=J^{(\ell)}$, respectively. By Proposition
\ref{prop-flat} applied first to $A=J^{(\ell)}$ and then to $A=E$, one gets
\begin{equation} \label{i-j}
{\rm Im}(\kappa_p)=\prod_{\fr p\mid p}j_\fr p\Big(H^1_{fl}\bigl(\cl O_{F_\fr p},\cl J^{(\ell)}_\fr p[p]\bigr)\Big),\qquad {\rm Im}(\delta_p)=\prod_{\fr p\mid p}i_\fr p\Big(H^1_{fl}\bigl(\cl O_{F_\fr p},\cl E_\fr p[p]\bigr)\Big).
\end{equation}
Since $p\in\fr m_{f_\ell}$ and ${\rm Ta}_p(J^{(\ell)})\big/p{\rm Ta}_p(J^{(\ell)})=J^{(\ell)}[p]$, the isomorphism in \eqref{rep-iso} gives a commutative triangle
\begin{equation} \label{triangle-tate-eq}
\xymatrix{{\rm Ta}_p\bigl(J^{(\ell)}\bigr)\ar@{->>}[d]\ar@{->>}[r]&E[p]\\
          J^{(\ell)}[p]\ar^-{\pi}@{->>}[ur]&}
\end{equation}
with surjective maps. For all primes $\fr p|p$ the generic fibers of $\cl J^{(\ell)}_\fr p[p]$ and $\cl E_\fr p[p]$ are $J^{(\ell)}_{/F_\fr p}[p]$ and $E_{/F_\fr p}[p]$ respectively, and $e<p-1$, so by \cite[Corollary 3.3.6]{Ray} the morphism $\pi=\pi_\fr p$ of \eqref{triangle-tate-eq} (viewed over $F_\fr p$) lifts uniquely to a morphism $\tilde\pi_\fr p:\cl J_\fr p^{(\ell)}[p]\rightarrow \cl E_\fr p[p]$. Hence for all primes $\fr p|p$ there is a commutative square
\begin{equation} \label{square-flat-eq}
\xymatrix{H^1_{fl}\bigl(\cl O_{F_\fr p},\cl J^{(\ell)}_\fr p[p]\bigr)\ar[r]^-{j_\fr p}\ar[d]^-{\tilde\pi_\fr p} & H^1\bigl(F_\fr p,J^{(\ell)}[p]\bigr)\ar[d]^-{\pi_\fr p}\\
             H^1_{fl}\bigl(\cl O_{F_\fr p},\cl E_\fr p[p]\bigr)\ar[r]^-{i_\fr p} & H^1\bigl(F_\fr p,E[p]\bigr)}
\end{equation}
in which the vertical maps are obtained functorially from $\tilde\pi_\fr p$ and $\pi_\fr p$, respectively. Finally, set $\pi_p:=\prod_{\fr p|p}\pi_\fr p$ and define $\bar\kappa_p:=\pi_p\circ\kappa_p$. Then if $P\in J^{(\ell)}(F)$ one has the formula
\[ \res_p\bigl(\bar\kappa(P)\bigr)=\bar\kappa_p(P). \]
On the other hand, the inclusion
\[ {\rm Im}(\bar\kappa_p)\subset\prod_{\fr p\mid p}i_\fr p\Big(H^1_{fl}\bigl(\cl O_{F_\fr p},\cl E_\fr p[p]\bigr)\Big)
={\rm Im}(\delta_p) \]
follows from \eqref{i-j} and \eqref{square-flat-eq}, and the proposition is proved. \end{proof}

\subsection{Heegner points on Shimura curves}

The Shimura curve $X^{(\ell)}$ introduced in Section \ref{section-Galois-representations} has a moduli interpretation which can be described as follows. As above, let $\cl B/\Q$ be the quaternion algebra of discriminant $N^-\ell$ and $\cl R\subset\cl B$ the Eichler order of level $N^+$ defining $X^{(\ell)}$. Fix a maximal order $\cl R_{\rm max}$ containing $\cl R$. Let $\cl F$ denote the functor from the category of $\Z[1/N\ell]$-schemes to the category of sets which associates to a $\Z[1/N\ell]$-scheme $S$ the set of isomorphism classes of triples $(A,\iota,C)$ where
\begin{enumerate}
\item $A$ is an abelian $S$-scheme of relative dimension $2$;
\item $\iota:\cl R_{\rm max}\rightarrow\End(A)$ is an inclusion defining an action on $A$ of $\cl R_{\rm max}$;
\item $C$ is a subgroup scheme of $A$ which is locally isomorphic to $\Z/N^+\Z$ and is stable and locally cyclic for the action of $\cl R_{\rm max}$.
\end{enumerate}
The functor $\cl F$ is coarsely representable by a smooth projective scheme with smooth fibers $X^{(\ell)}_{\Z[1/N\ell]}$ defined over $\Z[1/N\ell]$, whose generic fiber is $X^{(\ell)}$. See \cite[Ch. III]{BC} for details.

Let $P_c\in X^{(\ell)}$ correspond to a triple $(A_c,\iota_c,C_c)$ such that $\End(A_c)\simeq\cl O_c$. Such a point is called a \emph{Heegner point of conductor $c$}. By the theory of complex multiplication, $P_c\in X^{(\ell)}(H_c)$.

\section{The \v{C}erednik--Drinfeld theorem} \label{section-CD-theorem}

Let $\ell$ be an admissible prime relative to $f$ and $p$. Let $X_{\Z_\ell}^{(\ell)}$ be a nodal model of $X^{(\ell)}$ over
$\Z_\ell$, i.e. a proper and flat $\Z_\ell$-scheme such that
\begin{enumerate}
\item the generic fiber of $X_{\Z_\ell}^{(\ell)}$ is the base change $X^{(\ell)}\times_\Q\Q_\ell$;
\item the irreducible components of the special fiber $X_{\F_\ell}^{(\ell)}$ of $X_{\Z_\ell}^{(\ell)}$ are smooth;
\item the only singularities of $X_{\F_\ell}^{(\ell)}$ are ordinary double points.
\end{enumerate}
Before describing the result of \v{C}erednik and Drinfeld, we introduce a certain graph attached to $X_{\F_\ell}^{(\ell)}$.

\subsection{Dual graph of $X_{\F_\ell}^{(\ell)}$ and reduction map} \label{dual-graph-subsec}

The dual graph $\cl G_\ell$ of $X_{\F_\ell}^{(\ell)}$ is defined by requiring that its vertices $\cl V(\cl G_\ell)$ correspond to the irreducible components of $X_{\F_\ell}^{(\ell)}$, its edges $\cl E(\cl G_\ell)$ correspond to the singular points of $X_{\F_\ell}^{(\ell)}$ and two vertices are joined by an edge if and only if the corresponding components meet. Define a reduction map
\[ r_\ell:\Div\bigl(X^{(\ell)}(K_\ell)\bigr)\longrightarrow\Z\bigl[\cl V(\cl G_\ell)\cup\cl E(\cl G_\ell)\bigr] \]
as follows. Let $P\in X^{(\ell)}(K_\ell)$. Then $P$ can be extended to a point $\tilde P\in X_{\Z_\ell}^{(\ell)}(\cl O_{K_\ell})$, where $\cl O_{K_\ell}$ is the ring of integers of $K_\ell$. This corresponds to extending a triple $(A,\iota,C)$ over $K_\ell$ to a similar triple $(\tilde A,\tilde\iota,\tilde C)$ over $\cl O_{K_\ell}$. Denote by $\bar P$ the reduction of $\tilde P$ to the special fiber, which corresponds to the reduction $(\bar A,\bar\iota,\bar C)$ of $(\tilde A,\tilde\iota,\tilde C)$ modulo $\ell$. Then define $r_\ell(P)$ by requiring  that $r_\ell(P)$ is equal to a vertex $v$ (respectively, to an edge $e$) if $\bar P$ is non-singular and belongs to the component corresponding to $v$ (respectively, is the singular point corresponding to $e$).

Denote by $\Z[\cl V(\cl G_\ell)]^0$ the subgroup of divisors of degree $0$ in $\Z[\cl V(\cl G_\ell)]$. As explained in \cite[Corollary 5.12]{BD}, there is a natural map of groups
\[ \omega_\ell:\Z[\cl V(\cl G_\ell)]^0\longrightarrow\Phi_\ell. \]
Let $D\in\Div^0\bigl(X^{(\ell)}(K_\ell)\bigr)$ be a divisor of degree zero such that every point $P$ in the support of $D$ is defined over $\Q_{\ell^2}$ and has non-singular reduction, so that $r_\ell(P)\in\cl V(\cl G_\ell)$. Denote by $[D]$ the class of $D$ in $J^{(\ell)}(\Q_{\ell^2})$ and by
\begin{equation} \label{partial-eq}
\partial_\ell:J^{(\ell)}(\Q_{\ell^2})\longrightarrow\Phi_\ell
\end{equation}
the specialization map. By \cite[Proposition 5.14]{BD}, there is an equality
\begin{equation} \label{edix}
\partial_\ell([D])=\omega_\ell\bigl(r_\ell(D)\bigr).
\end{equation}

\subsection{The theorem of \v{C}erednik and Drinfeld}

Let $\cl T_\ell$ be the Bruhat-Tits tree of $\PGL_2(\Q_\ell)$, and denote $\cl V(\cl T_\ell)$ (respectively, $\cl E(\cl T_\ell)$) the set of vertices (respectively, edges) of $\cl T_\ell$.

Let $\widehat{\cl H}_\ell$ be the formal scheme over $\Z_\ell$ which is obtained by blowing up the projective line $\mathbb
P^1_{/\Z_\ell}$ along the rational points in its special fiber $\mathbb P_{/\F_\ell}^1$ successively (see \cite[Ch. I]{BC} for equivalent definitions of $\widehat{\cl H}_\ell$). The generic fiber $\cl H_\ell$ of $\widehat{\cl H}_\ell$ is a rigid analytic space over $\Z_\ell$ whose $\mathbb C_\ell$-points are given by
\[ \cl H_\ell(\C_\ell)=\mathbb P^1(\C_\ell)-\mathbb P^1(\Q_\ell)=\mathbb C_\ell-\Q_\ell. \]
By the theory of Schottky groups (see \cite{GvP} or \cite{BC} for an exposition), the connected components of the special fiber $\overline{\cl H}_{\ell}$ of $\widehat{\cl H}_\ell$ are smooth and meet transversely at ordinary double points. Hence we can form the dual graph of $\overline{\cl H}_{\ell}$, whose vertices are in bijection with the connected components of $\overline{\cl H}_\ell$, whose edges are in bijection with the singular points of $\overline{\cl H}_\ell$ and where two vertices are joined by an edge if and only if the corresponding components meet. This dual graph is identified with $\cl T_\ell$.

Let $\Gamma_{\ell}$ denote the group of norm $1$ elements in $R[1/\ell]$. Fix an isomorphism $i_\ell:B_\ell:=B\otimes_\Q\Q_\ell\rightarrow\M_2(\Q_\ell)$ such that $i_\ell(R\otimes\Z_\ell)=\M_2(\Z_\ell)$. The group $\Gamma_\ell$ acts discontinuously on $\widehat{\cl H}_\ell$ and $\cl H_\ell$ by fractional linear transformations via $i_\ell$. The quotients $\widehat{\cl H}_\ell/\Gamma_{\ell}$ and $\cl H_\ell/\Gamma_\ell$ are, respectively, a formal $\Z_\ell$-scheme and an $\ell$-adic rigid analytic space.

The \v{C}erednik-Drinfeld theorem (see \cite[III, Théorème 5.2]{BC}) says that the formal completion $\widehat X^{(\ell)}_{\Z_\ell}$ of $X^{(\ell)}_{\Z_\ell}$ along its special fiber is isomorphic over $\Z_{\ell^2}$ to $\widehat{\cl H}_\ell/\Gamma_{\ell}$ as a formal scheme. Hence the analytic space $X^{(\ell),\rm an}_{\Q_\ell}$ over $\Q_\ell$ associated to $X^{(\ell)}$ is isomorphic over $\Z_{\ell^2}$ to $\cl H_\ell/\Gamma_\ell$ as a rigid analytic space and its special fiber $X_{\F_\ell}^{(\ell)}$ is isomorphic over $\F_{\ell^2}$ to the special fiber of $\widehat{\cl H}_\ell/\Gamma_{\ell}$ as schemes. In particular, the dual graph of $X_{\F_\ell}^{(\ell)}$ is equal to the quotient of the dual graph of the special fiber of $\widehat{\cl H}_\ell$, that is, $\cl G_\ell\simeq \cl T_\ell/\Gamma_\ell$. It follows from this (see \cite[Lemma 5.6]{BD} for details) that there are identifications
\begin{equation} \label{vertexes-edges}
\cl V(\cl G_\ell)=\widehat R^\times\backslash\widehat B^\times/B^\times\times\{0,1\},\qquad \cl E(\cl G_\ell)=\widehat R_\ell^\times\backslash\widehat B^\times/B^\times
\end{equation}
where $R_\ell$ is the Eichler order of level $N^+\ell$ in $R$ such that $i_\ell(R_\ell)$ is the group of matrices in $\M_2(\Z_\ell)$ which are upper triangular modulo $\ell$. In particular, $\cl V(\cl G_\ell)$ identifies with the disjoint union of two copies of $\widehat R^\times\backslash\widehat B^\times/B^\times$.

\section{An explicit reciprocity law}

\subsection{Reduction of Heegner points}

Let the Heegner point $P_c$ correspond to the triple $(A_c,\iota_c,C_c)$ and write $\End(P_c)$ for the endomorphism ring of the abelian surface $A_c$ with its level structure, which is isomorphic to $\cl O_c$. Let $\ell$ be an admissible prime; the prime $\lambda_0$ of $H_c$ above $\ell$ chosen in \S \ref{admissible-subsec} determines an embedding 
\[ \iota_{\lambda_0}:H_c\;\longmono\;H_{c,\lambda_0}=K_\ell. \] 
The point $P_c$ can then be viewed as a point in $X^{(\ell)}(K_\ell)$ and it is possible to consider its image $\bar P_c$ in the special fiber as in Section \ref{section-CD-theorem}, which corresponds to a triple $(\bar A_c,\bar\iota_c,\bar C_c)$. Write $\End(\bar P_c)$ for the endomorphism ring of $\bar P_c$. By \cite[Lemma 4.13]{BD-inv-131}, $\End(\bar P_c)[1/\ell]$ is isomorphic to $R[1/\ell]$ where $R$ is the order introduced in Section \ref{section-GZ-formula}. Hence the canonical map $\End(A_c)\rightarrow \End(\bar A_c)$ obtained by reduction of endomorphisms can be extended to a map
\[ \varphi:\cl O_c[1/\ell]\simeq\End(A_c)\otimes_\Z\Z[1/\ell]\longrightarrow \End(\bar P_c)\otimes_\Z\Z[1/\ell]\simeq R[1/\ell]. \]
After tensoring with $\Q$ over $\Z[1/\ell]$, this yields an embedding
\[ \varphi:K\;\longmono\;B, \]
denoted by the same symbol. Let $\{R=R_1,\dots,R_h\}$ be a complete set of representatives for the isomorphism classes of Eichler orders of level $N^+$ in $B$. By \cite[Proposition 7.3]{GZ}, the map $\varphi:K\hookrightarrow B$ thus obtained is an optimal embedding of $\cl O_c$ into $R_i$ for some $i\in\{1,\dots,h\}$. Since, by the strong approximation theorem, this set of representatives can be chosen so that $R_i[1/p]=R[1/p]$ for all $i=1,\dots,h$, we can assume without loss of generality that $R_i=R$.

The group $B_\ell^\times$ acts on $\cl H_\ell$ by fractional linear transformations via $\iota_\ell$. Hence there is an action of $K_\ell^\times$ on $\cl H_\ell$ induced by extending $\varphi$ to an embedding $\varphi_\ell:K_\ell\hookrightarrow B_\ell$. By \cite[Section 4, III]{BD-inv-131}, the point $P_c\in X^{(\ell)}(K_p)$ is identified via the \v{C}erednik-Drinfeld Theorem with one of the two fixed points of the action of $K_\ell$ on $\cl H_\ell$ thus obtained.

\begin{lemma} \label{lemma-reduction}
The point $P_c\in X^{(\ell)}(H_c)$ reduces to a non-singular point in the special fiber.
\end{lemma}

\begin{proof} The Heegner point $P_c$ corresponds to a fixed point for the action of $\varphi_\ell(K_\ell^\times)$ on $\cl H_\ell$. On the other hand, $\varphi_\ell(\cl O_c\otimes\Z_\ell)$ is contained in exactly one maximal order of $B_\ell$ because $\cl O_c\otimes\Z_\ell$ is maximal and $\ell$ is inert in $K$, hence there are no optimal embeddings of $\cl O_c\otimes\Z_\ell$ into Eichler orders which are not maximal. Thus the action by conjugation of $\varphi_\ell(K_\ell^\times)$ on $\cl V(\cl T_\ell)\cup \cl E(\cl T_\ell)$ fixes exactly one vertex, namely the one corresponding to the order $R\otimes\Z_\ell$. By the $\GL_2(\Q_\ell)$-equivariance of $r_\ell$, we conclude that the reduction of $P_c$ must coincide with that vertex. \end{proof}

\subsection{$p$-isolated forms} \label{section-p-isolated}

Let $\cl S_2(B,R;\Z_p)$ denote the $\Z_p$-module of functions
\[ \widehat R^\times\backslash\widehat B^\times/B^\times\longrightarrow\Z_p. \] 
Then the function $\phi=\phi(f)$ associated to $f$ by the Jacquet--Langlands correspondence belongs to $\cl S_2(B,R;\Z_p)$. Now denote by $\fr m_f$ the kernel of $\bar f:\T(N)\rightarrow\Z/p\Z$. Say that $f$ is \emph{$p$-isolated} if the completion of $\cl S_2(B,R;\Z_p)$ at $\fr m_f$ is a free $\Z_p$-module of rank one. As a consequence, if $f$ is $p$-isolated then there are no non-trivial congruences modulo $p$ between $\phi$ and other forms in $\cl S_2(B,R;\Z_p)$.

By condition 4 in Assumption \ref{ass}, $p$ does not divide the minimal degree of a modular parametrization $X_0(N)\rightarrow E$, and this implies, as remarked in the proof of \cite[Lemma 2.2]{BD}, that $f$ is $p$-isolated (see also \cite[Theorem 2.2]{ARS} for details).

Another consequence of \ref{ass} is that the dimension of the $\F_p$-vector space $\Z[\widehat R^\times\backslash\widehat
B^\times/B^\times]^0/\fr m_f$ is at most one. This can be proved as follows. Let $r$ be a prime number dividing $N^-$ (such an $r$ exists because $N^-$ is the product of an odd number of primes). Denote by $\cl B^\prime$ the indefinite quaternion algebra over $\Q$ of discriminant $N^-/r$ and let $\cl R^\prime\subset\cl B^\prime$ be an Eichler order of level $N^+r$. As in \S \ref{section-X-ell}, denote by $\cl R^{\prime\times}_1$ the group of elements of norm $1$ in $\cl R^\prime$. Fix an embedding $i_\infty^\prime:\cl B^\prime\rightarrow\M_2(\R)$ and write $\Gamma^\prime_\infty$ for the image of $\cl R_1^{\prime\times}$ in $\PGL_2(\R)$ via the composition of $i_\infty^\prime$ with the canonical projection $\GL_2(\R)\rightarrow\PGL_2(\R)$. As in \S \ref{section-X-ell}, the analytic quotient $\cl H_\infty/\Gamma_\infty^\prime$ has a natural structure of compact Riemann surface which, by \cite[Ch. 9]{Sh}, admits a model $X^{(r)}=X^{(r)}_\Q$ defined over $\Q$. The character group of the jacobian of $X^{(r)}$ at $r$ can be identified with $\Z[\widehat R^\times\backslash\widehat B^\times/B^\times]^0$ as explained, e.g., in \cite[\S 5.3 and \S 5.4]{BD}. Then an extension of \cite[Theorem 6.4]{R} to the context of Shimura curves shows that the dimension of $\Z[\widehat R^\times\backslash\widehat B^\times/B^\times]^0/\fr m_f$ over $\F_p$ is at most one. For more details, the reader is referred to \cite[Theorem 5.15]{BD}.

\subsection{Hecke operators and modular forms} \label{hecke-operators-subsec}

The action of the Hecke operators on the function $\phi$ associated to the modular form $f$ by the Jacquet--Langlands  correspondence can be represented as follows. Let $\tilde\phi$ be the $\Z$-extension of $\phi$ to a map $\Z[\widehat R^\times\backslash\widehat B^\times/B^\times]\rightarrow \Z$ and $i$ the natural injection $\widehat R^\times\backslash\widehat B^\times/B^\times\hookrightarrow \Z[\widehat R^\times\backslash\widehat B^\times/B^\times]$. Any Hecke operator $T\in\T(N)$ acts on $\Z[\widehat R^\times\backslash\widehat B^\times/B^\times]$ by correspondences via the usual double coset decomposition. Then
\[ T(\phi):=\tilde\phi\circ T\circ i:\widehat R^\times\backslash\widehat B^\times/B^\times\longrightarrow\Z. \]
For any prime $q\nmid N$ define $\eta_q:=T_q-(q+1)\in\T(N)$. Since $\deg\bigl(\eta_q(D)\bigr)=0$ for every $D\in\Z[\widehat R^\times\backslash\widehat B^\times/B^\times]$, we can define
\begin{equation} \label{j-map-eq}
j:=\eta_q\circ i:\widehat R^\times\backslash\widehat B^\times/B^\times\longrightarrow\Z\bigl[\widehat R^\times\backslash\widehat B^\times/B^\times\bigr]^0, 
\end{equation}
where $\Z[\widehat R^\times\backslash\widehat B^\times/B]^0$ is the subgroup of divisors of degree $0$ in $\Z[\widehat R^\times\backslash\widehat B^\times/B]$. As before, let $\ell$ be an admissible prime. From here on fix a prime $q\not=\ell$ such that
\[ a_q\not\equiv q+1\pmod{p}, \]
which can be done by the \v{C}ebotarev density theorem and the surjectivity of $\rho_{E,p}$. Finally, write
\[ \bar\phi:\widehat R^\times\backslash\widehat B^\times/B^\times\longrightarrow\Z/p\Z \]
for the reduction modulo $p$ of $\phi$; then $\eta_q\bigl(\bar\phi\bigr)$ is a non-zero multiple of $\bar\phi$ and $\eta_q\bigl(\bar\phi\bigr)=\tilde\phi\circ j\pmod{p}$, where $\tilde\phi$ denotes, by an abuse of notations, the restriction of $\tilde\phi$ to $\Z[\widehat R^\times\backslash\widehat B^\times/B^\times]^0$.

Recall the ideal $\fr m_f\subset\T(N)$ defined in \S \ref{section-p-isolated}. Since $\Z[\widehat R^\times\backslash\widehat B^\times/B^\times]^0$ is invariant under $\T(N)$, we can consider its quotient by $\fr m_f$. This vector space has dimension one over $\F_p$, as stated in the next
\begin{prop} \label{one-dim-prop}
The $\F_p$-vector space $\Z\bigl[\widehat R^\times\backslash\widehat B^\times/B^\times\bigr]^0\big/\fr m_f$ has dimension one.
\end{prop}
\begin{proof} Since $\eta_q\bigl(\bar\phi\bigr)\neq 0$, the composition
\begin{equation} \label{*}
\widehat R^\times\backslash\widehat B^\times/B^\times\stackrel{j}{\longrightarrow}\Z\bigl[\widehat R^\times\backslash\widehat B^\times/B^\times\bigr]^0\;\longepi\;\Z\bigl[\widehat R^\times\backslash\widehat B^\times/B^\times\bigr]^0\big/\fr m_f
\end{equation}
is not zero. This shows in particular that the last quotient is not zero. On the other hand, its $\F_p$-dimension is at most one by the results explained in \S  \ref{section-p-isolated}, and the proposition follows. \end{proof}

Let $v_0$ be the vertex of $\cl T_\ell$ corresponding to the maximal order $\M_2(\Z_\ell)$. A vertex of $\cl T_\ell$ is said to be \emph{even} (respectively, \emph{odd}) if its distance from $v_0$ is even (respectively, odd). Since the elements of $\Gamma_\ell$ have determinant $1$, there is a well-defined notion of even and odd vertices in the graph $\cl G_\ell$ of \S \ref{dual-graph-subsec}. Define maps
\begin{equation} \label{orientation-G-ell-eq}
s,t:\cl E(\cl G_\ell)\longrightarrow\cl V(\cl G_\ell) 
\end{equation}
by requiring that $s(e)$ is the even vertex and $t(e)$ is the odd vertex of the edge $e$. By \eqref{vertexes-edges}, every $v\in\cl V(\cl G_\ell)$ can be regarded as a pair 
\begin{equation} \label{v-b-j-eq}
v=(b_v,i) 
\end{equation}
with $b_v\in\widehat R^\times\backslash\widehat B^\times/B^\times$ and $i\in\{0,1\}$. Moreover, as explained in the proof of \cite[Lemma 5.6]{BD}, we can assume that $i=0$ if and only if $v$ is even. Then define 
\[ \delta_*:\Z[\cl E(\cl G_\ell)]^0\longrightarrow\Z[\cl V(\cl G_\ell)]^0 \]
to be the restriction to $\Z[\cl E(\cl G_\ell)]^0$ of the $\Z$-linear map sending an edge $e$ to $t(e)-s(e)$. Observe that, with the above identifications, the map $\delta_*$ can be written as 
\[ \delta_*(e)=(b_{t(e)},1)-(b_{s(e)},0). \] 
The submodule ${\rm Im}(\delta_*)$ can be identified with the product of two copies of $\Z[\widehat R^\times\backslash\widehat B^\times/B^\times]^0$. Let us briefly review why this is true. The orientation on the edges of $\cl G_\ell$ chosen in \eqref{orientation-G-ell-eq} induces two maps
\[ \alpha_*, \beta_*:\Z[\cl E(\cl G_\ell)]\longrightarrow\Z[\widehat R^\times\backslash\widehat B^\times/B^\times]\times\{0,1\} \]
defined by extending $\Z$-linearly the rules $\alpha_\ast(e):=(b_{t(e)},1)$ and $\beta_\ast(e):=(-b_{s(e)},0)$. By restriction, $\alpha_\ast$ and $\beta_*$ give maps
\[ \alpha_*^0:\Z[\cl E(\cl G_\ell)]^0\longrightarrow\Z[\widehat R^\times\backslash\widehat B^\times/B^\times]^0\times\{1\},\qquad \beta_*^0:\Z[\cl E(\cl G_\ell)]^0\longrightarrow\Z[\widehat R^\times\backslash\widehat B^\times/B^\times]^0\times\{0\}. \]
Finally, set
\[ d_*:=(\alpha_*^0,\beta_*^0):\Z[\cl E(\cl G_\ell)]^0\longrightarrow\bigl(\Z[\widehat R^\times\backslash\widehat B^\times/B^\times]^0\times\{1\}\bigr)\times\bigl(\Z[\widehat R^\times\backslash\widehat B^\times/B^\times]^0\times\{0\}\bigr). \]
To ease the notation, denote the codomain of $d_*$ simply by $\bigl(\Z[\widehat R^\times\backslash\widehat B^\times/B^\times]^0\bigr)^2$. As remarked in \cite[Proposition 5.7]{BD}, it can be checked that $d_*$ is surjective and that the diagram
\[ \xymatrix@C=30pt{\Z[\cl E(\cl G_\ell)]^0\ar@{=}[d]\ar[r]^-{\delta_*}&\Z[\cl V(\cl G_\ell)]^0\ar@{^{(}->}[r]&\Z[\widehat R^\times\backslash\widehat B^\times/B^\times]^2\\
   \Z[\cl E(\cl G_\ell)]^0\ar@{->>}[rr]^-{d_\ast}&&\bigl(\Z[\widehat R^\times\backslash\widehat B^\times/B^\times]^0\bigr)^2\ar@{^{(}->}[u]} \]
is commutative, and this shows that there is an identification
\begin{equation} \label{Im-delta*-eq}
{\rm Im}(\delta_*)=\bigl(\Z[\widehat R^\times\backslash\widehat B^\times/B^\times]^0\times\{1\}\bigr)\times\bigl(\Z[\widehat R^\times\backslash\widehat B^\times/B^\times]^0\times\{0\}\bigr)\simeq\bigl(\Z[\widehat R^\times\backslash\widehat B^\times/B^\times]^0\bigr)^2. 
\end{equation}
Let $\T_\ell^\old$ be the quotient of $\T(N\ell)$ acting on forms which are old at $N^-$. It follows from the description of $\cl E(\cl G_\ell)$ in \eqref{vertexes-edges} that $\T_\ell^\old$ acts on $\Z[\cl E(\cl G_\ell)]$ by correspondences. By \cite[Proposition 5.8]{BD}, the quotient ${\rm Im}(\delta_*)$ of $\cl E(\cl G_\ell)$ is stable under the action of $\T_\ell^\old$. Denote $U'_\ell$ the $\ell$-operator in $\T_\ell^\old$, so $U'_\ell$ acts on ${\rm Im}(\delta_*)$. Then, by \cite[Section 7]{R} (see also \cite[Theorems 3.12, 4.3 and 5.2 (c)]{R} and \cite[Theorem 5.15]{BD}), the restriction of $\omega_\ell$ to ${\rm Im}(\delta_*)$ gives rise to an isomorphism of groups
\begin{equation} \label{omega}
\bar\omega_\ell:{\rm Im}(\delta_*)\big/\big\langle\fr m_f,(U'_\ell)^2-1\big\rangle\overset{\simeq}{\longrightarrow}\Phi_\ell/\fr m_{f_\ell}.
\end{equation}
Here recall that $f_\ell$ is the modular form introduced in \eqref{f-ell}, which is new at $\ell$ and congruent to $f$ modulo $p$, while $\fr m_{f_\ell}$ is its associated maximal ideal of residual characteristic $p$. By \cite[Theorem 3.19]{R} (see also
\cite[Theorem 5.15]{BD}), the action of $U'_\ell$ on ${\rm Im}(\delta_*)$ is given by $(x,y)\mapsto(T_\ell x-y,\ell x)$.
Therefore, since
\[ a_\ell\equiv\epsilon(\ell+1)\pmod{p} \]
with $\epsilon\in\{\pm1\}$, it follows that $U'_\ell+\epsilon$ is invertible on ${\rm Im}(\delta_*)$ while the image of $U'_\ell-\epsilon$ is the subset
\[ \Big\{(\epsilon x,x)\;\big|\;x\in\Z\bigl[\widehat R^\times\backslash\widehat B^\times/B^\times\bigr]^0\Big\}\subset{\rm Im}(\delta_*). \]
Combining this with Proposition \ref{one-dim-prop} and the isomorphism in \eqref{omega}, and recalling that ${\rm Im}(\delta_*)$ identifies with the product of two copies of $\Z[\widehat R^\times\backslash\widehat B^\times/B^\times]^0$, shows that
\begin{equation} \label{Phi-ell-2}
\Phi_\ell/\fr m_{f_\ell}\simeq\Z/p\Z.
\end{equation}
With $j$ as in \eqref{j-map-eq}, note that ${\rm Im}(j)$ can naturally be viewed as a submodule of ${\rm Im}(\delta_*)$ in two ways, either via $x\mapsto (x,0)$ or via $x\mapsto (0,x)$. Choose one of the two maps above, call it $\tilde\iota$ and denote by $\iota$ the map obtained by composing $\tilde\iota$ with the canonical projection, that is
\begin{equation} \label{***}
\iota:{\rm Im}(j)\overset{\tilde\iota}{\longrightarrow}{\rm Im}(\delta_*)\;\longepi\;{\rm Im}(\delta_*)\big/\big\langle\fr m_f,(U'_\ell)^2-1\big\rangle.
\end{equation}
Then we obtain a map
\begin{equation}\label{**}
\bar\omega_\ell\circ\iota\circ j:\widehat R^\times\backslash\widehat B^\times/B^\times\longrightarrow\Phi_\ell/\fr m_{f_\ell}\simeq\Z/p\Z.
\end{equation}
The above description of $\bigl((U'_\ell)^2-1\bigr){\rm Im}(\delta_*)$ and the fact that ${\rm Im}(j)/\fr m_f$ is not
trivial (because the map \eqref{*} is not) show that the map \eqref{**} is non-zero. Hence, since $f$ is $p$-isolated, we
conclude that $\bar\omega_\ell\circ\iota\circ j$ is equal to $\bar\phi$ up to multiplication by a constant in $(\Z/p\Z)^\times$.

\subsection{The explicit reciprocity law} \label{reciprocity-subsec}

By Lemma \ref{lemma-reduction}, $v_c:=r_\ell(P_c)$ is in $\cl V(\cl G_\ell)$, hence (with notation as in \eqref{v-b-j-eq}) it can be identified with a pair $(b_{v_c},i)$ in $\widehat R^\times\backslash\widehat B^\times/B^\times\times\{0,1\}$. By \cite[Proposition 5.8]{BD}, the action of $\eta_q$ on ${\rm Im}(\delta_*)$ is diagonal with respect to the decomposition described in \eqref{Im-delta*-eq}, so we can write
\[ \eta_q(v_c)=\bigl(\eta_q(b_{v_c}),i\bigr)\in\Z[\widehat R^\times\backslash\widehat B^\times/B^\times]^0\times\{i\} \]
with the same $i$ as before. Choose $\tilde\iota$ in \eqref{***} such that $\eta_q(v_c)\in{\rm Im}(\tilde\iota)$. Let
\[ \bar\partial_\ell:J^{(\ell)}(K_\ell)\longrightarrow\Phi_\ell/\fr m_{f_\ell}\simeq\Z/p\Z \]
be the composition of the map $\partial_\ell$ introduced in \eqref{partial-eq} with the canonical projection modulo $\fr m_{f_\ell}$ and the isomorphism \eqref{Phi-ell-2}. The operator $\eta_q$ defines a Hecke and Galois-equivariant map
\[ \eta_q:\Div(X^{(\ell)})\longrightarrow\Div^0(X^{(\ell)}). \]
For any $D\in\Div^0(X^{(\ell)})$ write $[D]$ for the class of $D$ in $J^{(\ell)}$, and set
\[ \xi_c:=[\eta_q(P_c)]\in J^{(\ell)}(H_c). \]
Define
\[ \alpha_c:=\sum_{\sigma\in G_c}\xi_c^\sigma\otimes\sigma^{-1}\in J^{(\ell)}(H_c)\otimes_\Z\Z[G_c]. \]
Recall the embedding $i_{\lambda_0}:H_c\hookrightarrow H_{c,\lambda_0}$ associated to the prime $\lambda_0$. We remark in
passing that, since $\ell$ splits completely in $H_c$, one has $i_\lambda=\sigma_\lambda\circ i_{\lambda_0}$ for all primes
$\lambda|\ell$ (here the Galois elements $\sigma_\lambda$ are as in \eqref{bijection-inverse-eq}). By an abuse of notation, denote in the same way the global-to-local homomorphism
\[ i_{\lambda_0}:J^{(\ell)}(H_c)\;\longmono\;J^{(\ell)}(H_{c,\lambda_0}) \]
induced on the points of the jacobian, and keep in mind that for all primes $\lambda|\ell$ there are canonical identifications $H_{c,\lambda}=K_\ell=\Q_{\ell^2}$. Viewing the maps $i_\lambda$ as $K_\ell$-valued via these identifications (which we will sometimes do without explicit warning) one then has $i_\lambda=i_{\lambda_0}$ for all primes $\lambda|\ell$.

The above discussion combined with \eqref{edix} yields the equality
\begin{equation} \label{edix-II}
\bigl(\bar\partial_\ell\circ i_{\lambda_0}\bigr)(\xi_c)=(\bar\omega_\ell\circ\iota\circ j)(v_c)\in\Z/p\Z.
\end{equation}
As a piece of notation, for any ring $A$ and any pair of elements $a,b\in A$ let $a\doteq b$ mean that there exists $c\in A^\times$ such that $a=bc$. Moreover, let $[\star]$ denote the class of the element $\star$ in a quotient group.

\begin{teo} \label{rec-law}
The relation
\[ \bigl((\bar\partial_\ell\circ i_{\lambda_0})\otimes\chi\bigr)(\alpha_c)\doteq\bigl[\cl L(f,\chi)\bigr] \]
holds in $\Z[\chi]\big/p\Z[\chi]$.
\end{teo}

\begin{proof} Since the maps $\bar\omega_\ell\circ\iota\circ j$ and $\bar\phi $ are equal up to multiplication by an element in $(\Z/p\Z)^\times$, the result follows from \eqref{edix-II} and the definition of $\cl L(f,\chi)$, after noticing that $r_\ell$ is Galois-equivariant. \end{proof}

\section{Algebraic preliminaries} \label{algebraic-sec}

\subsection{Towards the vanishing of ring class eigenspaces}

Let $\chi\in\widehat G_c$ be our complex-valued character of $G_c$. The prime $p$ is unramified in $\Z[\chi]$ since it does not divide $h(c)$ by condition $1$ in Assumption \ref{ass}. Choose a prime ideal $\fr p$ of $\Z[\chi]$ above $p$ such that
\begin{equation} \label{alg-nonzero-eq}
\text{the image of $\cl L(f,\chi)$ in $\Z[\chi]/\fr p$ is not zero.}
\end{equation}
This can be done thanks to condition $3$ in Assumption \ref{ass}. Denote $\W$ the $\fr p$-adic completion of $\Z[\chi]$ and observe that the non-ramification of $p$ implies that $p\W$ is the maximal ideal of $\W$; in particular,
\[ \Z[\chi]/\fr p=\W/p\W. \]
For any $\Z[G_c]$-module $M$ write $M\otimes_\chi\C$ (respectively, $M\otimes_\chi\W$) for the tensor product of the $\Z[G_c]$-modules $M$ and $\C$ (respectively, $M$ and $\W$), where the structure of $\Z[G_c]$-module on $\C$ (respectively, $\W$) is induced by $\chi$.

As in the introduction, if $M$ is a $\Z[G_c]$-module define
\[ M^\chi:=\bigl\{x\in M\otimes_\Z\mathcal\C\mid\text{$\sigma(x)=\chi(\sigma)x$ for all $\sigma\in G_c$}\bigr\}. \]
The next elementary algebraic result (of which we sketch a proof for lack of an explicit reference) will be repeatedly used in the sequel.

\begin{prop} \label{alg-1}
If $M$ is a $\Z[G_c]$-module which is finitely generated as an abelian group then there is a canonical identification
\[ M^\chi=M\otimes_\chi\C \]
of $\C[G_c]$-modules.
\end{prop}
\begin{proof} By the universal property of tensor products, there is a natural $G_c$-equivariant surjection
\[ F:M\otimes_\Z\C\longrightarrow M\otimes_\chi\C \]
of finite-dimensional $\C$-vector spaces. Moreover, Maschke's theorem ensures that $M\otimes_\Z\C$ decomposes as a direct sum
\[ M\otimes_\Z\C=\bigoplus_{\gamma\in\widehat G_c}M^\gamma \]
of primary representations, and $F$ induces a $G_c$-equivariant (surjective) map
\[ f=F|_{M^\chi}:M^\chi\longrightarrow M\otimes_\chi\C. \]
Finally, again by universality one obtains a natural $G_c$-equivariant map
\[ g:M\otimes_\chi\C\longrightarrow M^\chi \]
of $\C$-vector spaces, and it can be checked that $f$ and $g$ are inverses of each other. \end{proof}
Choose once and for all an (algebraic) isomorphism $\C_p\simeq\C$ which is the identity on $\Z[\chi]$. Henceforth we shall view $\C$ as a $\W$-module via this isomorphism, obtaining an isomorphism
\begin{equation} \label{tensor-can-eq}
\bigl(E(H_c)\otimes_\chi\W\bigr)\otimes_\W\C\simeq E(H_c)\otimes_\chi\C.
\end{equation}
For later reference, we state the following

\begin{lemma} \label{alg-2}
The module $\W$ is flat over $\Z[G_c]$, and every $\F_p[G_c]$-module is flat.
\end{lemma}

\begin{proof} First of all, $\W$ is flat over $\Z$. Moreover, if $\ell$ is a prime number dividing $h(c)$ then $\ell\not=p$, hence $\W/\ell\W=0$. The flatness of $\W$ follows from \cite[Theorem 1.6]{BG}. The second assertion can be shown in exactly the same way. \end{proof}

Now we can prove

\begin{prop} \label{alg-4}
If $\Sel_p(E/H_c)\otimes_\chi\W=0$ then $E(H_c)^\chi=0$.
\end{prop}

\begin{proof} By Proposition \ref{alg-1} and equation \eqref{tensor-can-eq}, it is enough to show that
\[ E(H_c)\otimes_\chi\W=0. \]
Being finitely generated over $\Z$, the module $E(H_c)$ is \emph{a fortiori} finitely generated over $\Z[G_c]$, hence $E(H_c)\otimes_\chi\W$ is finitely generated as a $\W$-module. But $p\W$ is the maximal ideal of $\W$, so Nakayama's lemma ensures that the vanishing of $E(H_c)\otimes_\chi\W$ is equivalent to the vanishing of
\[ \bigl(E(H_c)\otimes_\chi\W\bigr)\otimes_\W(\W/p\W)\simeq E(H_c)\otimes_\chi(\W/p\W)\simeq\bigl(E(H_c)/pE(H_c)\bigr)\otimes_\chi\W. \]
By Lemma \ref{alg-2}, $\W$ is flat over $\Z[G_c]$. Tensoring the $G_c$-equivariant injection $E(H_c)/pE(H_c)\hookrightarrow\Sel_p(E/H_c)$ with $\W$ over $\Z[G_c]$ then yields an injection
\[ \bigl(E(H_c)/pE(H_c)\bigr)\otimes_\chi\W\;\longmono\;\Sel_p(E/H_c)\otimes_\chi\W. \]
Since $\Sel_p(E/H_c)\otimes_\chi\W=0$ by assumption, the proposition is proved. \end{proof}

Thus the triviality of $E(H_c)^\chi$ is guaranteed by that of $\Sel_p(E/H_c)\otimes_\chi\W$. This vanishing will be proved in Theorem \ref{th-sel-n}. To this end, we need a couple of further algebraic lemmas, to which the next $\S$ is devoted.

\subsection{Towards the vanishing of Selmer groups}

In the following, we adopt the same symbol $\chi$ to denote the $\Z$-linear extension
\[ \Z[G_c]\overset{\chi}{\longrightarrow}\Z[\chi]\subset\W \]
of the character $\chi$. If we compose $\chi$ with the projection onto $\W/p\W$ then the resulting homomorphism factors through $\F_p[G_c]=\Z[G_c]/p\Z[G_c]$, and the triangle
\[ \xymatrix@R=35pt{\Z[G_c]\ar[r]^-\chi\ar@{->>}[d]&\W\ar@{->>}[r]&\W/p\W\\
          \F_p[G_c]\ar[urr]^-{\chi_p}&&} \]
is commutative. In particular, the map $\chi_p$ gives $\W/p\W$ a structure of $\F_p[G_c]$-module (which is nothing but the structure induced naturally by that of $\Z[G_c]$-module on $\W$), and for an $\F_p[G_c]$-module $M$ the notation $M\otimes_{\chi_p}(\W/p\W)$ will be used to indicate that the tensor product is taken over $\F_p[G_c]$ via $\chi_p$.

Write $I_{\chi_p}$ for the kernel of $\chi_p$, and for any $\F_p[G_c]$-module $M$ let
\[ M[I_{\chi_p}]:=\bigl\{m\in M\mid \text{$xm=0$ for all $x\in I_{\chi_p}$}\bigr\} \]
be the $I_{\chi_p}$-torsion submodule of $M$. Moreover, let $\chi_p^{-1}:\F_p[G_c]\rightarrow\W/p\W$ be the map induced by the inverse character to $\chi$. Equivalently, for any $a=\sum_\sigma a_\sigma\sigma\in\F_p[G_c]$ set $a^{-1}:=\sum_\sigma a_\sigma\sigma^{-1}$. Since $G_c$ is abelian, the map $\sigma\mapsto\sigma^{-1}$ is an automorphism of $G_c$ which induces an algebra automorphism $\varpi$ of $\F_p[G_c]$ sending $a$ to $a^{-1}$. Then there is an equality of maps
\[ \xymatrix@C=30pt{\F_p[G_c]\ar[r]_-{\simeq}^-{\varpi}\ar@/^1.7pc/[rr]^-{\chi_p^{-1}} & \F_p[G_c]\ar[r]^-{\chi_p} & \W/p\W.} \]
Finally, write $I_{\chi_p^{-1}}$ for the kernel of $\chi_p^{-1}$ and for any $\F_p[G_c]$-module $M$ denote by $M\bigl[I_{\chi_p^{-1}}\bigr]$ the $I_{\chi_p^{-1}}$-torsion submodule of $M$.  

\begin{lemma} \label{lemma10.1}
For every $\F_p[G_c]$-module $M$ there are canonical identifications
\[ M\otimes_\chi\W=M\otimes_{\chi_p}(\W/p\W)=M[I_{\chi_p}]\otimes_{\chi_p}(\W/p\W)=M[I_{\chi_p}]\otimes_\chi\W \]
of $\W$-modules.
\end{lemma}

\begin{proof} Since $\F_p[G_c]$ is a (commutative) noetherian ring, we can choose generators $x_1,\dots,x_m$ of the ideal $I_{\chi_p}$. For all $i=1,\dots,m$ there is a short exact sequence
\[ 0\longrightarrow K_i:=\ker(\mu_{x_i})\longrightarrow M\xrightarrow{\mu_{x_i}}x_iM\longrightarrow 0, \]
where $\mu_{x_i}$ is the multiplication-by-$x_i$ map. Since $\W/p\W$ is flat over $\F_p[G_c]$ by Lemma \ref{alg-2}, tensoring the above sequence with $\W/p\W$ over $\F_p[G_c]$ gives an equality
\[ M\otimes_{\chi_p}(\W/p\W)=K_i\otimes_{\chi_p}(\W/p\W) \]
for all $i=1,\dots,m$. But $\cap_iK_i=M[I_{\chi_p}]$, hence
\[ \begin{split} 
M\otimes_{\chi_p}(\W/p\W)&=\bigcap_{i=1}^m\bigl(K_i\otimes_{\chi_p}(\W/p\W)\bigr)\\
&=\Big(\bigcap_{i=1}^mK_i\Big)\otimes_{\chi_p}(\W/p\W)=M[I_{\chi_p}]\otimes_{\chi_p}(\W/p\W), 
\end{split} \]
where the second equality is a consequence of the flatness of $\W/p\W$. 

Finally, a straightforward algebraic argument shows that there is an identification
\[ N\otimes_\chi\W=N\otimes_{\chi_p}(\W/p\W) \]
of $\W$-modules for every $\F_p[G_c]$-module $N$, and the lemma is completely proved. \end{proof}

\begin{lemma} \label{lemma-injection}
If $M$ is an $\F_p[G_c]$-module then $M[I_{\chi_p}]$ injects into $M\otimes_\chi\W$. Moreover, $M\bigl[I_{\chi^{-1}_p}\bigr]$ injects into $M\bigl[I_{\chi^{-1}_p}\bigr]\otimes_\chi\W$.
\end{lemma}

\begin{proof} Since $M[I_{\chi_p}]$ is flat over $\F_p[G_c]$ by Lemma \ref{alg-2}, tensoring the short exact sequence 
\begin{equation} \label{seq-eq}
0\longrightarrow I_{\chi_p}\longrightarrow\F_p[G_c]\overset{\chi_p}{\longrightarrow}\text{Im}(\chi_p)\longrightarrow0 
\end{equation}
with $M[I_{\chi_p}]$ over $\F_p[G_c]$ yields a short exact sequence
\[ 0\longrightarrow M[I_{\chi_p}]\otimes_{\F_p[G_c]}I_{\chi_p}\longrightarrow M[I_{\chi_p}]\longrightarrow M[I_{\chi_p}]\otimes_{\chi_p}\text{Im}(\chi_p)\longrightarrow0. \]
But the image of $M[I_{\chi_p}]\otimes I_{\chi_p}$ in $M[I_{\chi_p}]$ is trivial, so there is an isomorphism
\begin{equation} \label{inj1-eq}
M[I_{\chi_p}]\overset{\simeq}{\longrightarrow}M[I_{\chi_p}]\otimes_{\chi_p}\text{Im}(\chi_p). 
\end{equation}
Finally, the flatness of $M[I_{\chi_p}]$ gives a further injection
\begin{equation} \label{inj2-eq}
M[I_{\chi_p}]\otimes_{\chi_p}\text{Im}(\chi_p)\;\longmono\;M[I_{\chi_p}]\otimes_{\chi_p}(\W/p\W),
\end{equation} 
whence the first claim by combining \eqref{inj1-eq} and \eqref{inj2-eq} with Lemma \ref{lemma10.1}. 

The second assertion follows in a completely analogous way upon replacing \eqref{seq-eq} with 
\[ 0\longrightarrow I_{\chi^{-1}_p}\longrightarrow\F_p[G_c]\xrightarrow{\chi^{-1}_p}\text{Im}(\chi_p)\longrightarrow0 \]
and tensoring with $M\bigl[I_{\chi^{-1}_p}\bigr]$ over $\F_p[G_c]$. \end{proof}

\begin{rem} \label{injection-rem}
For future reference, we explicitly observe that the injections of Lemma \ref{lemma-injection} send $x$ to $x\otimes1$.
\end{rem}

For any $\F_p$-vector space $V$ denote the $\F_p$-dual of $V$ by 
\begin{equation} \label{dual-V-eq} 
V^\vee:=\Hom_{\F_p}(V,\F_p). 
\end{equation} 
If the map $f:V_1\rightarrow V_2$ of $\F_p$-vector spaces is injective (respectively, surjective) then the dual map $f^\vee:V_2^\vee\rightarrow V_1^\vee$ is surjective (respectively, injective). The dual of an $\F_p[G_c]$-module inherits a natural structure of $\F_p[G_c]$-module: a Galois element $\sigma$ acts on a homomorphism $\varphi$ by $\sigma(\varphi):=\varphi\circ\sigma^{-1}$. Furthermore, if $f$ is a map of $\F_p[G_c]$-modules then its dual $f^\vee$ is again $G_c$-equivariant.

\begin{rem} \label{I-chi-conj-rem}
It follows directly from the definition of the dual $G_c$-action that if an $\F_p[G_c]$-module $M$ is of $I_{\chi_p}$-torsion (i.e., $M=M[I_{\chi_p}]$) then its dual $M^\vee$ is of $I_{\chi_p^{-1}}$-torsion (i.e., $M^\vee=M^\vee[I_{\chi_p^{-1}}]$).
\end{rem} 

We apply the above results to the $I_{\chi_p}$-torsion submodule of the Selmer group $\Sel_p(E/H_c)$. Let $\ell$ be an admissible prime and let $C$ denote the cokernel of the restriction 
\[ \res_\ell:\Sel_p(E/H_c)[I_{\chi_p}]\longrightarrow H^1_{\rm fin}(H_{c,\ell},E[p])[I_{\chi_p}]. \]
Of course, $C=C[I_{\chi_p}]$. There is a commutative diagram with exact rows
\begin{equation} \label{I-chi-diag-eq}
\xymatrix@C=38pt{\Sel_p(E/H_c)[I_{\chi_p}]\ar[r]^-{\res_\ell}\ar@{^{(}->}[d]&H^1_{\rm fin}(H_{c,\ell},E[p])[I_{\chi_p}]\ar@{->>}[r]\ar@{^{(}->}[d]& C\ar@{^{(}->}[d]\\
\Sel_p(E/H_c)\otimes_\chi\W\ar[r]^-{\res_\ell\otimes\mathrm{id}}&H^1_{\rm fin}(H_{c,\ell},E[p])\otimes_\chi\W\ar@{->>}[r]& C\otimes_\chi\W} 
\end{equation}
in which the right horizontal arrows are surjective and the vertical arrows are ${\rm id}\otimes 1$. Observe that the bottom row is a consequence of Lemma \ref{lemma10.1} and the flatness of $\W$, while the vertical maps are injections by Lemma \ref{lemma-injection}.

\begin{lemma} \label{aux-lemma-I}
If there exists $s\in\Sel_p(E/H_c)[I_{\chi_p}]$ such that $\res_\ell(s)\neq0$ then the map 
\[ \res_\ell^\vee\otimes{\rm id}:H^1_{\rm fin}(H_{c,\ell},E[p])[I_{\chi_p}]^\vee\otimes_\chi\W\longrightarrow\Sel_p(E/H_c)[I_{\chi_p}]^\vee\otimes_\chi\W \] 
is injective and non-zero. 
\end{lemma} 

\begin{proof} Keep the notation of \eqref{I-chi-diag-eq} and let $s\in\Sel_p(E/H_c)[I_{\chi_p}]$ be such that $\res_\ell(s)\neq0$. Then Lemma \ref{lemma-injection} ensures that $\res_\ell(s)\otimes 1\neq0$ in $H^1_{\rm fin}(H_{c,\ell},E[p])\otimes_\chi\W$ (cf. Remark \ref{injection-rem}). Since 
\[ H^1_{\rm fin}(H_{c,\ell},E[p])\otimes_\chi\W\simeq\W/p\W \]
by Lemma \ref{local-iso-lemma}, the non-zero element $\res_\ell(s)\otimes1$ generates $H^1_{\rm fin}(H_{c,\ell},E[p])\otimes_\chi\W$ as a $\W/p\W$-vector space. Then the $\W$-linear map
\[ \res_\ell\otimes\text{id}:\Sel_p(E/H_c)\otimes_\chi\W\longrightarrow H^1_{\rm fin}(H_{c,\ell},E[p])\otimes_\chi\W \] 
is surjective, and this shows that $C\otimes_\chi\W=0$. It follows from the injection in diagram \eqref{I-chi-diag-eq} that $C=0$ as well, thus the map 
\[ \res_\ell:\Sel_p(E/H_c)[I_{\chi_p}]\longrightarrow H^1_{\rm fin}(H_{c,\ell},E[p])[I_{\chi_p}] \] 
is surjective. By duality, we obtain an injection 
\[ \res_\ell^\vee:H^1_{\rm fin}(H_c,E[p])[I_{\chi_p}]^\vee\;\longmono\;\Sel_p(E/H_c)[I_{\chi_p}]^\vee. \]  
Since $\W$ is a flat $\Z[G_c]$-module, tensoring with $\W$ gives the required injection
\begin{equation} \label{res-tens-inj-eq}
\res_\ell^\vee\otimes{\rm id}:H^1_{\rm fin}(H_{c,\ell},E[p])[I_{\chi_p}]^\vee\otimes_\chi\W\;\longmono\;\Sel_p(E/H_c)[I_{\chi_p}]^\vee\otimes_\chi\W. 
\end{equation}
Our assumption on $s$ ensures that $H^1_{\rm fin}(H_{c,\ell},E[p])[I_{\chi_p}]\not=0$, hence $H^1_{\rm fin}(H_{c,\ell},E[p])[I_{\chi_p}]^\vee\not=0$ as well. On the other hand, $H^1_{\rm fin}(H_{c,\ell},E[p])[I_{\chi_p}]^\vee$ is of $I_{\chi_p^{-1}}$-torsion (cf. Remark \ref{I-chi-conj-rem}), so it injects into $H^1_{\rm fin}(H_{c,\ell},E[p])[I_{\chi_p}]^\vee\otimes_\chi\W$ by Lemma \ref{lemma-injection}. This shows that
\[ H^1_{\rm fin}(H_{c,\ell},E[p])[I_{\chi_p}]^\vee\otimes_\chi\W\not=0, \]
and we conclude that the injection \eqref{res-tens-inj-eq} is non-zero. \end{proof}

As a consequence of the above results, we can prove 

\begin{prop} \label{lemma-II}
If there exists $s\in\Sel_p(E/H_c)[I_{\chi_p}]$ such that $\res_\ell(s)\neq0$ then the map 
\[ \res_\ell^\vee\otimes{\rm id}:H^1_{\rm fin}(H_{c,\ell},E[p])^\vee\otimes_\chi\W\longrightarrow\Sel_p(E/H_c)^\vee\otimes_\chi\W \] 
is non-zero.
\end{prop}
\begin{proof} In the commutative square
\[ \xymatrix@C=30pt{H^1_{\rm fin}(H_{c,\ell},E[p])^\vee\otimes_\chi\W\ar[r]^-{\res_\ell^\vee\otimes\text{id}}\ar@{->>}[d]&\Sel_p(E/H_c)^\vee\otimes_\chi\W\ar@{->>}[d]\\ 
H^1_{\rm fin}(H_{c,\ell},E[p])[I_{\chi_p}]^\vee\otimes_\chi\W\ar@{^{(}->}[r]&\Sel_p(E/H_c)[I_{\chi_p}]^\vee\otimes_\chi\W} \]
the vertical maps are surjective and the bottom horizontal arrow is (injective and) non-zero by Lemma \ref{aux-lemma-I}. Hence the upper horizontal arrow must be non-zero. \end{proof} 

\section{Bounding the Selmer group} \label{selmer-sec}

\subsection{An equality of divisors} \label{divisors-subsec}

Let $\ell$ be an admissible prime and, as usual, set $H_{c,\ell}:=H_c\otimes\Q_\ell=\oplus_{\lambda\mid\ell}H_{c,\lambda}$,
where the sum is over the primes of $H_c$ above $\ell$. An element of the completion $H_{c,\lambda}$ can be viewed as an equivalence class of Cauchy sequences in $H_c$ with respect to the $\lambda$-adic metric, two such sequences being identified if their difference has limit $0$. On the other hand, $H_{c,\lambda}=K_\ell$ for all $\lambda|\ell$. With this in mind, in this $\S$ we adopt the following notation: if $\lambda|\ell$ and $x\in H_c$ then we choose a sequence $\boldsymbol x_\lambda:=(x_{\lambda,n})_{n\in\N}$ of elements of $K$ converging to $x$ in the $\lambda$-adic topology, and denote $[\boldsymbol x_\lambda]$ its equivalence class in $K_\ell$. In other words, $x=[\boldsymbol x_\lambda]$ in $H_{c,\lambda}=K_\ell$.

Thus we obtain a composition of maps
\begin{equation} \label{Hc1}
\begin{array}{ccccc}H_c&\longmono&\bigoplus_{\lambda|\ell}H_{c,\lambda}&\overset{\simeq}{\longrightarrow}&\bigoplus_{\lambda|\ell}K_\ell\\[2mm]
x&\longmapsto&(x,\dots,x)&\longmapsto&\bigl([\boldsymbol x_\lambda]\bigr)_{\lambda|\ell}
\end{array}
\end{equation}
which gives a concrete description of the ``diagonal embedding'' of $H_c$ in $\oplus_{\lambda|\ell}K_\ell$.

Now recall the prime $\lambda_0$ above $\ell$ fixed in \S \ref{admissible-subsec}. Then, with notation as in \eqref{bijection-inverse-eq}, there is an identification of $K_\ell$-vector spaces between $\oplus_{\lambda\mid\ell}K_\ell$ and the group algebra $K_\ell[G_c]$ via the map
\begin{equation} \label{Hc2}
(y_\lambda)_{\lambda|\ell}\longmapsto\sum_{\lambda|\ell}y_\lambda\sigma_\lambda=\sum_{\sigma\in G_c}y_{\sigma(\lambda_0)}\sigma.
\end{equation}
Note that the map in \eqref{Hc2} depends on the choice of the prime ideal $\lambda_0$ above $\ell$, while the one in \eqref{Hc1} does not. Let now
\[ \theta_\ell:H_c\longrightarrow K_\ell[G_c] \]
be the composition of the maps in \eqref{Hc1} and \eqref{Hc2}; then $\theta_\ell$ depends on the choice of $\lambda_0$. Explicitly, one has
\begin{equation} \label{theta-eq}
\theta_\ell(x)=\sum_{\sigma\in G_c}[\boldsymbol x_{\sigma(\lambda_0)}]\sigma
\end{equation}
for all $x\in H_c$. By definition, for all $\sigma\in G_c$ the sequence $\boldsymbol x_{\sigma(\lambda_0)}$ converges to $x$ in the $\sigma(\lambda_0)$-adic topology. On the other hand, since $x_{\sigma(\lambda_0),n}\in K$ for all $n\in\N$, it also follows from this that $\boldsymbol x_{\sigma(\lambda_0)}$ converges to $\sigma^{-1}(x)$ in the $\lambda_0$-adic topology.

Let $\beta_\ell:H_c[G_c]\rightarrow K_\ell[G_c]$ denote the map defined by
\[ \beta_\ell\Big(\sum_{\sigma\in G_c}z_\sigma\sigma\Big):=\sum_{\sigma\in G_c}[\boldsymbol z_\sigma]\sigma \]
where $[\boldsymbol z_\sigma]$ is the equivalence class of a sequence $\boldsymbol z_\sigma:=(z_{\sigma,n})_{n\in\N}$ of elements of $K$ converging to $z_\sigma$ in the $\lambda_0$-adic topology (so $\beta_\ell$ depends on the choice of  $\lambda_0$). It is straightforward to check that
\begin{equation} \label{Hc4}
\theta_\ell(x)=\beta_\ell\Big(\sum_{\sigma\in G_c}\sigma^{-1}(x)\sigma\Big)
\end{equation}
for all $x\in H_c$. The above maps $\theta_\ell$ and $\beta_\ell$ yield in the obvious way maps
\[ \theta_\ell:J^{(\ell)}(H_c)\;\longmono\;J^{(\ell)}(K_\ell)\otimes\Z[G_c],\qquad \beta_\ell:J^{(\ell)}(H_c)\otimes\Z[G_c]\longrightarrow J^{(\ell)}(K_\ell)\otimes\Z[G_c] \]
which we denote by the same symbols.

With notation as in \S \ref{reciprocity-subsec}, one then has the following

\begin{lemma} \label{equality-P-alpha}
$\theta_\ell(\xi_c)=\beta_\ell(\alpha_c)$.
\end{lemma}
\begin{proof} Immediate from equality \eqref{Hc4} by definition of $\xi_c$. \end{proof}

\subsection{The Galois action}

Fix $\sigma_0\in G_c$, let $x\in H_c$ and set $y:=\sigma_0(x)\in H_c$. Then equation \eqref{theta-eq} shows that
\begin{equation} \label{Hc3}
\theta_\ell(y)=\sum_{\sigma\in G_c}[\boldsymbol y_{\sigma(\lambda_0)}]\sigma,
\end{equation}
where $\boldsymbol y_{\sigma(\lambda_0)}$ is a sequence of elements in $K$ which converges to $y$ in the $\sigma(\lambda_0)$-adic topology. Thus the sequence $\boldsymbol y_{\sigma(\lambda_0)}$ also converges to $x$ in the $\sigma_0^{-1}\sigma(\lambda_0)$-adic topology. On the other hand, there is a canonical action of $\sigma_0$ on $K_\ell[G_c]$ induced by multiplication on group-like elements and given by
\[ \sigma_0\Big(\sum_{\sigma\in G_c}z_\sigma\sigma\Big):=\sum_{\sigma\in G_c}z_\sigma(\sigma_0\sigma). \]
Then there is an equality
\[ \sigma_0\bigl(\theta_\ell(x)\bigr)=\sum_{\sigma\in G_c}[\boldsymbol x_{\sigma(\lambda_0)}](\sigma_0\sigma)=\sum_{\sigma\in G_c}\bigl[\boldsymbol x_{\sigma_0^{-1}\sigma(\lambda_0)}\bigr]\sigma. \]
In the above equation, the sequence $\boldsymbol x_{\sigma_0^{-1}\sigma(\lambda_0)}$ converges to $x$ in the $\sigma_0^{-1}\sigma(\lambda_0)$-adic topology. This fact and the observation immediately after \eqref{Hc3} show that $\theta_\ell\bigl(\sigma_0(x)\bigr)=\sigma_0\bigl(\theta_\ell(x)\bigr)$. Since $\sigma_0\in G_c$ and $x\in H_c$ were taken arbitrarily, it follows that
\begin{equation} \label{theta-sigma}
\theta_\ell\circ\sigma=\sigma\circ\theta_\ell
\end{equation}
for all $\sigma\in G_c$. In other words, the map $\theta_\ell$ is $G_c$-equivariant.
\begin{rem}
Equality \eqref{theta-sigma} holds both as a relation between maps from $H_c$ to $K_\ell[G_c]$ and as a relation between maps from $J^{(\ell)}(H_c)$ to $J^{(\ell)}(K_\ell)\otimes\Z[G_c]$.
\end{rem}

\subsection{Construction of an Euler system}

Let
\begin{equation} \label{d-ell-eq}
d_\ell:H^1(H_c,E[p])\longrightarrow H_{\rm sing}^1(H_{c,\ell},E[p]) 
\end{equation}
denote the composition of $\text{res}_\ell$ with the projection onto the singular quotient. We introduce $\chi$-twisted versions of the maps $\theta_\ell$, $\beta_\ell$ and $d_\ell$ defined above. Namely, set
\begin{equation} \label{twisted-maps-eq}
\begin{split}
\theta_\ell^\chi&:=(\text{id}\otimes\chi)\circ\theta_\ell:J^{(\ell)}(H_c)\longrightarrow J^{(\ell)}(K_\ell)\otimes_\Z\W,\\
\beta_\ell^\chi&:=(\text{id}\otimes\chi)\circ\beta_\ell:J^{(\ell)}(H_c)\otimes_\Z\Z[G_c]\longrightarrow J^{(\ell)}(K_\ell)\otimes_\Z\W,\\
d_\ell^\chi&:=d_\ell\otimes1:H^1(H_c,E[p])\longrightarrow H_{\rm sing}^1(H_{c,\ell},E[p])\otimes_\chi\W.
\end{split}
\end{equation}
As explained in Lemma \ref{local-iso-lemma}, the choice of the prime $\lambda_0$ of $H_c$ above $\ell$ induces isomorphisms
\[ \gamma_\ell:H^1_\star(H_{c,\ell},E[p])\overset{\simeq}{\longrightarrow}H^1_\star(K_\ell,E[p])\otimes_\Z\Z[G_c]=\Z/p\Z[G_c] \]
for $\star\in\{\text{fin, sing}\}$ (the maps $\gamma_\ell$ are defined as in \S \ref{divisors-subsec}, and the equality on the right is a consequence of the identifications in \eqref{fin-sing-iso-eq}). Since $\lambda_0$ has been fixed once and for all in \S \ref{admissible-subsec}, we will sometimes regard the maps $\gamma_\ell$ as identifications (or, better, canonical isomorphisms). Then it is easily checked that the square
\begin{equation} \label{square-aux-eq}
\xymatrix{H^1(H_c,E[p])\ar[r]^-{\gamma_\ell\circ d_\ell}\ar[d]^-{d_\ell^\chi} & H^1_{\rm sing}(K_\ell,E[p])\otimes_\Z\Z[G_c]\ar[d]^-{{\rm id}\otimes\chi}\\
             H^1_{\rm sing}(H_{c,\ell},E[p])\otimes_\chi\W\ar[r]^-{\simeq} & H^1_{\rm sing}(K_\ell,E[p])\otimes_\Z\W}
\end{equation}
is commutative (here the bottom row is given by $\gamma_\ell$ followed by the canonical identification between tensor products).

In light of diagram \eqref{square-aux-eq}, there is a further commutative diagram
\begin{equation} \label{diagram2}
\xymatrix{J^{(\ell)}(H_c)\ar[r]^-{\theta_\ell^\chi}\ar[d]^-{\bar\kappa} & J^{(\ell)}(K_\ell)\otimes_\Z\W\ar[r]^-{\bar\partial_\ell\otimes{\rm id}} & \Phi_\ell/\fr m_{f_\ell}\otimes_\Z\W\ar[d]^-{\nu_\ell}_-{\simeq} \\
H^1(H_c,E[p])\ar[r]^-{d_\ell^\chi} & H_{\rm sing}^1(H_{c,\ell},E[p])\otimes_\chi\W\ar[r]^-{\simeq} & H^1_{\rm sing}(K_\ell,E[p])\otimes_\Z\W }
\end{equation}
where $\bar\kappa$ is defined as in \eqref{kappa} and $\nu_\ell$ is the isomorphism of \cite[Corollary 5.18]{BD} tensored with the identity map of $\W$ (see \emph{loc. cit.} for details). Now define a cohomology class
\[ \kappa(\ell):=\bar\kappa(\xi_c)\in H^1(H_c,E[p]). \]
\begin{prop} \label{nonvanishing-prop}
If $L_K(f,\chi,1)\neq0$ then
\[ d_\ell^\chi\bigl(\kappa(\ell)\bigr)\neq0 \]
in $H^1_{\rm sing}(H_{c,\ell},E[p])\otimes_\chi\W$.
\end{prop}
\begin{proof} By Lemma \ref{equality-P-alpha} we know that $\theta_\ell(\xi_c)=\beta_\ell(\alpha_c)$, hence
\[ \theta_\ell^\chi(\xi_c)=\beta_\ell^\chi(\alpha_c). \]
Keeping in mind that $L_K(f,\chi,1)\not=0$, assumption \eqref{alg-nonzero-eq} ensures that the image of $\cl L(f,\chi)$ in $\W/p\W$ is non-zero, and then $\bigl(\bar\partial_\ell\otimes{\rm id}\bigr)\bigl(\beta_\ell^\chi(\alpha_c)\bigr)\not=0$ by Theorem \ref{rec-law}. Now the proposition follows by the commutativity of \eqref{diagram2} and the definition of $\kappa(\ell)$. \end{proof}

The collection of cohomology classes $\{k(\ell)\}$ indexed by the set of admissible primes is an \emph{Euler system} relative to $E_{/K}$ and will be used in the sequel to bound the $p$-Selmer groups.

\subsection{Local and global Tate pairings}

For any prime $q$ of $\Z$ denote by
\[ \langle\,,\rangle_q:H^1(H_{c,q},E[p])\times H^1(H_{c,q},E[p])\longrightarrow\Z/p\Z \]
the local Tate pairing at $q$ and by
\[ \langle\,,\rangle:H^1(H_c,E[p])\times H^1(H_c,E[p])\longrightarrow\Z/p\Z \]
the global Tate pairing. The reciprocity law of class field theory shows that
\begin{equation} \label{GRL}
\langle k,s\rangle=\sum_q\langle \res_q(k),\res_q(s)\rangle_q=0
\end{equation}
for all $k,s\in H^1(H_c,E[p])$. From here on let $\ell$ be an admissible prime. The local Tate pairing $\langle\,,\rangle_\ell$ gives rise to a non-degenerate pairing of finite dimensional $\F_p$-vector spaces
\[ \langle\,,\rangle_\ell:H^1_{\rm fin}(H_{c,\ell},E[p])\times H^1_{\rm sing}(H_{c,\ell},E[p])\longrightarrow \Z/p\Z; \]
then, with notation as in \eqref{dual-V-eq}, if $\{\star,\bullet\}=\{\text{fin, sing}\}$ we get isomorphisms of finite-dimensional $\F_p$-vector spaces
\begin{equation} \label{isom-tate-eq}
H^1_\star(H_{c,\ell},E[p])\overset{\simeq}{\longrightarrow}H^1_\bullet(H_{c,\ell},E[p])^\vee.
\end{equation}
It is immediate to see that
\[ \langle\sigma(k),\sigma(s)\rangle_\ell=\langle k,s\rangle_\ell \]
for all $\sigma\in G_c$, $k\in H^1_{\rm fin}(H_{c,\ell},E[p])$ and $s\in H^1_{\rm sing}(H_{c,\ell},E[p])$, hence the map \eqref{isom-tate-eq} is $G_c$-equivariant (as usual, $\sigma$ acts on a homomorphism $\varphi$ by $\sigma(\varphi):=\varphi\circ\sigma^{-1}$).

Now recall that, by Lemma \ref{local-iso-lemma}, both the finite and the singular cohomology at $\ell$ are isomorphic to $\Z/p\Z[G_c]$ as $\Z[G_c]$-modules. Then from \eqref{isom-tate-eq} we get isomorphisms of one-dimensional $\W/p\W$-vector spaces
\begin{equation} \label{iso-tate}
H^1_\star(H_{c,\ell},E[p])\otimes_\chi\W\overset{\simeq}{\longrightarrow}H^1_\bullet(H_{c,\ell},E[p])^\vee\otimes_\chi\W.
\end{equation}
By local Tate duality, the restriction
\[ \res_\ell:\Sel_p(E/H_c)\longrightarrow H^1_{\rm fin}(H_{c,\ell},E[p]) \]
induces a $\W$-linear map 
\[ \eta_\ell:H^1_{\rm sing}(H_{c,\ell},E[p])\otimes_\chi\W\longrightarrow\Sel_p(E/H_c)^\vee\otimes_\chi\W. \]

\begin{lemma} \label{aux-lemma-IV} 
If there exists $s\in\Sel_p(E/H_c)[I_{\chi_p}]$ such that $\res_\ell(s)\neq 0$ then $\eta_\ell$ is non-zero.
\end{lemma}
\begin{proof} Immediate from \eqref{iso-tate} and Proposition \ref{lemma-II}. \end{proof}

For the next two results, recall the maps $d_\ell$ and $d_\ell^\chi$ defined in \eqref{d-ell-eq} and \eqref{twisted-maps-eq}, and recall also that $p$ satisfies Assumption \ref{ass}. In the first statement we use the notation of \S \ref{classical-selmer-subsec}.

\begin{lemma} \label{bad-kummer-lemma}
If $q$ is a prime dividing $N$ then ${\rm Im}(\delta_q)=0$.
\end{lemma}

\begin{proof} Since the local Kummer map $\delta_q$ factors through $E(H_{c,q})/pE(H_{c,q})$, the claim follows from condition $5$ in Assumption \ref{ass}. \end{proof}

The following is essentially a reformulation of \cite[Lemma 6.4]{BD}.

\begin{prop} \label{kernel-eta-prop} 
The element $d_\ell^\chi\bigl(\kappa(\ell)\bigr)$ belongs to the kernel of $\eta_\ell$.
\end{prop}

\begin{proof} 
Let $s\in\Sel_p(E/H_c)$. Since $p\geq5$ is unramified in $H_c$ (because $p\nmid cD$ by condition $1$ in Assumption \ref{ass}), we can apply Proposition \ref{prop-kummer} to $P=\xi_c$ and get
\[ \big\langle\res_q(s),\res_q\bigl(\kappa(\ell)\bigr)\big\rangle_q=0 \]
for all primes $q\neq N\ell$. On the other hand, Lemma \ref{bad-kummer-lemma} says that ${\rm Im}(\delta_q)=0$ for all primes $q|N$, hence by definition of the Selmer group we conclude that
\[ \big\langle\res_q(s),\res_q\bigl(\kappa(\ell)\bigr)\big\rangle_q=0 \]
for these finitely many primes $q$ as well. Then equation \eqref{GRL} implies that
\[ \big\langle\res_\ell(s),\res_\ell\bigl(\kappa(\ell)\bigr)\big\rangle_\ell=0,\] 
which shows that $d_\ell\bigl(\kappa(\ell)\bigr)$ belongs to the kernel of the map 
\begin{equation} \label{dual-sel-eq}
H^1_{\rm sing}(H_{c,\ell},E[p])\longrightarrow\Sel_p^\vee(E/H_c) 
\end{equation} 
induced by the local Tate duality. By definition of $d_\ell^\chi$ and $\eta_\ell$, tensoring \eqref{dual-sel-eq} with $\W$ gives the result. \end{proof}

\subsection{Proof of the main result}

Now we are in a position to (re)state and prove the main result of our paper, that is the part of Theorem \ref{thm-intro-2} concerning the Selmer groups: all other results will follow from this. 

\begin{teo} \label{th-sel-n}
If $L_K(f,\chi,1)\neq 0$ then $\Sel_p(E/H_c)\otimes_\chi\W=0$.
\end{teo}

\begin{proof} By Lemma \ref{lemma10.1}, it is enough to show that $\Sel_p(E/H_c)[I_{\chi_p}]=0$. Assume that $s\in\Sel_p(E/H_c)[I_{\chi_p}]$ is not zero. Choose an admissible prime $\ell$ such that $\res_\ell(s)\neq0$, whose existence is guaranteed by Proposition \ref{existence-admissible-primes}. Since $L_K(f,\chi,1)\not=0$, Proposition \ref{nonvanishing-prop} ensures that $d_\ell^\chi\bigl(\kappa(\ell)\bigr)$ is non-zero; then $d_\ell^\chi\bigl(\kappa(\ell)\bigr)$ generates $H^1_{\rm sing}(H_{c,\ell},E[p])\otimes_\chi\W$ over $\W$. On the other hand, Proposition \ref{kernel-eta-prop} says that $d_\ell^\chi\bigl(\kappa(\ell)\bigr)$ belongs to the kernel of the $\W$-linear map $\eta_\ell$, and this contradicts the non-triviality of $\eta_\ell$ that was shown in Lemma \ref{aux-lemma-IV}. \end{proof}

Now recall that for all $n\geq1$ the $p^n$-Shafarevich--Tate group $\Sha_{p^n}(E/H_c)$ of $E$ over $H_c$ is the cokernel of the Kummer map $\delta$, so that it fits into a short exact sequence of $G_c$-modules
\begin{equation} \label{sha-eq}
0\longrightarrow E(H_c)/p^nE(H_c)\overset{\delta}{\longrightarrow}\Sel_{p^n}(E/H_c)\longrightarrow \Sha_{p^n}(E/H_c)\longrightarrow 0.
\end{equation}
Equivalently, define
\[ \Sha(E/H_c):=\Ker\Big(H^1\bigl(H_c,E(\bar\Q)\bigr)\xrightarrow{\prod_q\res_q}\prod_q H^1\bigl(H_{c,q},E(\bar\Q)\bigr)\Big); \]
then $\Sha_{p^n}(E/H_c)$ is the $p^n$-torsion subgroup of $\Sha(E/H_c)$. As usual, set
\[ \Sel_{p^\infty}(E/H_c):=\dirlim_n\Sel_{p^n}(E/H_c),\qquad \Sha_{p^\infty}(E/H_c):=\dirlim_n\Sha_{p^n}(E/H_c), \]
so $\Sha_{p^\infty}(E/H_c)$ is the $p$-primary subgroup of $\Sha(E/H_c)$. Tensoring \eqref{sha-eq} with $\W$ over $\Z[G_c]$,
and using the fact that $\W$ is flat over $\Z[G_c]$ by Lemma \ref{alg-2}, we get a short exact sequence
\begin{equation} \label{sha-eq-II}
0\rightarrow\bigl(E(H_c)/p^nE(H_c)\bigr)\otimes_\chi\W\rightarrow\Sel_{p^n}(E/H_c)\otimes_\chi\W\rightarrow\Sha_{p^n}(E/H_c)\otimes_\chi\W\rightarrow 0.
\end{equation}
As a by-product of Theorem \ref{th-sel-n}, we also obtain

\begin{teo} \label{sel-sha-vanishing-teo}
If $L_K(f,\chi,1)\neq 0$ then
\[ \Sel_{p^n}(E/H_c)\otimes_\chi\W=0, \qquad \Sha_{p^n}(E/H_c)\otimes_\chi\W=0 \]
for all integers $n\geq1$.
\end{teo}

\begin{proof}
The case $n=1$ is immediate from Theorem \ref{th-sel-n} and \eqref{sha-eq-II}. Since $\W$ is flat over $\Z[G_c]$, for all $n\geq1$ there is an equality
\[ \bigl(\Sha(E/H_c)\otimes_\chi\W\bigr)[p^n]=\Sha_{p^n}(E/H_c)\otimes_\chi\W. \]
Then the vanishing of $\Sha_p(E/H_c)\otimes_\chi\W$ implies that of $\Sha_{p^n}(E/H_c)\otimes_\chi\W$ for all $n\geq1$, and the statement about the Shafarevich--Tate groups is proved.

To prove the vanishing of the Selmer groups one can proceed as follows. By condition $2$ in Assumption \ref{ass}, the representation $\rho_{E,p}$ is surjective, hence the $p$-torsion of $E(H_c)$ is trivial (\cite[Lemma 4.3]{G2}). Thus for all $n\geq1$ there is a short exact sequence of $G_c$-modules
\[ 0\longrightarrow E(H_c)\overset{p^n}{\longrightarrow}E(H_c)\longrightarrow E(H_c)/p^nE(H_c)\longrightarrow 0, \]
where the second arrow is the multiplication-by-$p^n$ map. The flatness of $\W$ over $\Z[G_c]$ then shows that
\begin{equation} \label{E-p^n-eq}
\bigl(E(H_c)/p^nE(H_c)\bigr)\otimes_\chi\W\simeq\bigl(E(H_c)\otimes_\chi\W\bigr)\big/p^n\bigl(E(H_c)\otimes_\chi\W\bigr).
\end{equation}
But \eqref{sha-eq-II} implies that $\bigl(E(H_c)/pE(H_c)\bigr)\otimes_\chi\W=0$ because $\Sel_p(E/H_c)\otimes_\chi\W=0$ by Theorem \ref{th-sel-n}, hence \eqref{E-p^n-eq} immediately gives
\[ \bigl(E(H_c)/p^nE(H_c)\bigr)\otimes_\chi\W=0 \]
for all $n\geq1$. Since we already know that $\Sha_{p^n}(E/H_c)\otimes_\chi\W=0$, now sequence \eqref{sha-eq-II} yields
\[ \Sel_{p^n}(E/H_c)\otimes_\chi\W=0 \]
for all $n\geq1$, and this completes the proof of the theorem. \end{proof}

Since tensor product commutes with direct limits, the following result is an immediate consequence of Theorem \ref{sel-sha-vanishing-teo}.

\begin{coro}
If $L_K(f,\chi,1)\neq 0$ then
\[ \Sel_{p^\infty}(E/H_c)\otimes_\chi\W=0,\qquad \Sha_{p^\infty}(E/H_c)\otimes_\chi\W=0. \]
\end{coro}

\subsection{Applications} \label{applications-subsec}

The first consequence of Theorem \ref{th-sel-n} is Theorem \ref{thm-intro}, which we now restate.

\begin{teo}[Bertolini--Darmon] If $L_K(f,\chi,1)\neq 0$ then $E(H_c)^\chi=0$. \end{teo}

\begin{proof} Immediate upon combining Proposition \ref{alg-4} and Theorem \ref{th-sel-n}. \end{proof}

As remarked in the introduction, this is the $\chi$-twisted conjecture of Birch and Swinnerton-Dyer for $E$ over $H_c$ in the case of analytic rank zero.

Now we want to show how, by specializing Theorem \ref{thm-intro-2} to the trivial character, one can prove the finiteness of $E(K)$ and obtain vanishing results for almost all $p$-Selmer groups of $E$ \emph{over} $K$. 

Recall that the conjecture of Shafarevich and Tate (ST conjecture, for short) predicts that if $E_{/F}$ is an elliptic defined over a number field $F$ then the Shafarevich--Tate group $\Sha(E/F)$ of $E$ over $F$ is finite. As pointed out (for real quadratic fields) in \cite[Theorem 4.1]{bdd}, the next theorem is a consequence of Kolyvagin's results on Euler systems of Heegner points.

\begin{teo}[Kolyvagin] \label{kol-teo}
Let $E_{/\Q}$ be an elliptic curve and let $K$ be a quadratic number field. If $L_K(E,1)\not=0$ then $E(K)$ and $\Sha(E/K)$ are finite.
\end{teo}
\begin{proof}[Sketch of proof.] There is an equality of $L$-series
\[ L_K(E,s)=L(E,s)\cdot L(E_{(\varepsilon)},s) \]
where $E_{(\varepsilon)}$ is the twist of $E$ by the Dirichlet character $\varepsilon$ attached to $K$, hence both $L(E,1)$ and $L(E_{(\varepsilon)},1)$ are non-zero. By analytic results of Bump, Friedberg and Hoffstein (\cite{bfh}) and Murty and Murty (\cite{mm}) there exist auxiliary imaginary quadratic fields $K_1,K_2$ such that 
\begin{itemize}
\item $L'_{K_1}(E,1)\not=0$, $L'_{K_2}(E_{(\varepsilon)},1)\not=0$;
\item all primes dividing the conductor of $E$ (respectively, of $E_{(\varepsilon)}$) split in $K_1$ (respectively, in $K_2$). 
\end{itemize}
Then Kolyvagin's theorem (see, e.g., \cite{G2}, \cite{ko} and \cite{ru}) shows that $E(\Q)$, $E_{(\varepsilon)}(\Q)$, $\Sha(E/\Q)$ and $\Sha(E_{(\varepsilon)}/\Q)$ are finite, and this implies that $E(K)$ and $\Sha(E/K)$ are finite as well (see \cite[Corollary B]{ko} for details). \end{proof}

In other words, both the BSD conjecture and the ST conjecture are true for $E_{/K}$ as in the statement of the theorem. The consequence of Theorem \ref{kol-teo} we are interested in is the following

\begin{coro} \label{kol-coro}
Let $E_{/\Q}$ be an elliptic curve and let $K$ be a quadratic number field. If $L_K(E,1)\not=0$ then $\Sel_p(E/K)=\Sha_p(E/K)=0$ for all but finitely many primes $p$.
\end{coro}

Let now $E_{/\Q}$ be an elliptic curve, let $K$ be an imaginary quadratic field and suppose that all the arithmetic assumptions made at the outset of this article are satisfied. Denote by $\boldsymbol 1=\boldsymbol 1_{G_c}$ the trivial character of the Galois group $G_c$. As anticipated a few lines above, we specialize Theorem \ref{thm-intro-2} to the case where $\chi=\boldsymbol 1$ and deduce results on the $p$-Selmer groups of $E$ over $K$. More precisely, in Theorem \ref{sel-K-vanishing-teo} we provide an alternative proof of Corollary \ref{kol-coro} for $E$ over $K$ which does not rely on either finiteness results for Shafarevich--Tate groups or Kolyvagin's results in rank one. 

Thus take $c=1$, $\chi=\boldsymbol 1_{G_1}$ and, for simplicity, set $H:=H_1$, $G:=G_1$. Observe that $H$ is the Hilbert class field of $K$ and $G\simeq\Pic(\cl O_K)$. Since $\chi$ is trivial, one has $\Z[\chi]=\Z$ and $\W=\Z_p$. Moreover, there is an equality 
\[ L_K(f,\boldsymbol 1_G,s)=L_K(E,s) \]
of $L$-functions. 

Before proving the main result of this section, we need a couple of auxiliary facts.

\begin{lemma} \label{sel-p-inj-lem}
There is an injection
\[ \Sel_p(E/K)\;\longmono\;\Sel_p(E/H)^G \]
for all prime numbers $p$ satisfying condition 2 in Assumption \ref{ass}.
\end{lemma}
\begin{proof} As explained in \cite[Lemma 4.3]{G2}, condition 2 in Assumption \ref{ass} ensures that $E$ has no $p$-torsion rational over $H$. Hence the group $H^1\bigl(G,E[p](H)\bigr)$ is trivial, and the desired injection is a consequence of the inflation-restriction sequence in Galois cohomology. \end{proof}

The group $G$ acts trivially on $\Sel_p(E/H)^G$ (by definition) and on $\Z_p$ (since $\chi=\boldsymbol 1_G$). The next lemma is an exercise in linear algebra.

\begin{lemma} \label{sel-tensor-lem}
There is an isomorphism
\[ \Sel_p(E/H)^G\otimes_\Z\Z_p\simeq\Sel_p(E/H)^G\otimes_{\Z[G]}\Z_p \]
of $\Z_p$-modules. 
\end{lemma}
\begin{proof} As a shorthand, set $M:=\Sel_p(E/H)^G$. To avoid confusion, denote $m\otimes x$ and $m\otimes'x$ the images of a pair $(m,x)\in M\times\Z_p$ in $M\otimes_\Z\Z_p$ and $M\otimes_{\Z[G]}\Z_p$, respectively. The map
\[ M\times\Z_p\longrightarrow M\otimes_{\Z[G]}\Z_p,\qquad (m,x)\longmapsto m\otimes'x \]
is clearly $\Z$-bilinear, hence it induces a map
\[ M\otimes_\Z\Z_p\overset{\varphi}{\longrightarrow}M\otimes_{\Z[G]}\Z_p,\qquad m\otimes x\longmapsto m\otimes' x \]
of abelian groups. Analogously, since $G$ acts trivially, there is a map
\[ M\otimes_{\Z[G]}\Z_p\overset{\psi}{\longrightarrow}M\otimes_\Z\Z_p,\qquad m\otimes'x\longmapsto m\otimes x \]
of $\Z[G]$-modules. Of course, the maps $\varphi$ and $\psi$ are inverse of each other, hence we get an isomorphism
\[ M\otimes_\Z\Z_p\simeq M\otimes_{\Z[G]}\Z_p \]
of abelian groups (actually, of $\Z[G]$-modules, but we will not need this fact). Since $\varphi$ and $\psi$ are $\Z_p$-linear, the lemma is proved. \end{proof}

Now we can prove the following result, which is a consequence of Theorem \ref{thm-intro-2} and was stated in the introduction as Theorem \ref{sel-K-intro-teo} and in this $\S$ as Corollary \ref{kol-coro}.

\begin{teo} \label{sel-K-vanishing-teo} 
If $L_K(E,1)\not=0$ then $\Sel_p(E/K)=\Sha_p(E/K)=0$ for all but finitely many primes $p$ and $E(K)$ is finite.
\end{teo}
\begin{proof} Of course, it suffices to prove the vanishing of the $p$-Selmer groups. Suppose that $p$ satisfies Assumption \ref{ass}. By Theorem \ref{thm-intro-2}, we know that $\Sel_p(E/H)\otimes_{\Z[G]}\Z_p=0$, hence
\begin{equation} \label{sel-inv-tensor-eq}
\Sel_p(E/H)^G\otimes_{\Z[G]}\Z_p=0 
\end{equation} 
since $\Z_p$ is flat over $\Z[G]$ by Lemma \ref{alg-2}. Combining \eqref{sel-inv-tensor-eq} and Lemma \ref{sel-tensor-lem}, we get
\begin{equation} \label{sel-inv-tensor-eq2}
\Sel_p(E/H)^G\otimes_\Z\Z_p=0.
\end{equation}
But $\Sel_p(E/H)^G$ is a finite dimensional $\F_p$-vector space, so \eqref{sel-inv-tensor-eq2} implies that 
\[ \Sel_p(E/H)^G=0. \]
The vanishing of $\Sel_p(E/K)$ follows from Lemma \ref{sel-p-inj-lem}. \end{proof}

\begin{rem}
1. As is clear from the proof, to obtain Theorem \ref{sel-K-vanishing-teo} one can fix an arbitrary $c$ prime to $ND$, as in the rest of the paper: we chose $c=1$ only to avoid excluding more primes $p$ than necessary for the purposes of this section.

2. The BSD conjecture for $E$ over $K$ in the case of analytic rank $0$, which was already proved in Theorem \ref{kol-teo},  can also be easily deduced from Theorem \ref{thm-intro} by specialization to the trivial character of $G_c$.
\end{rem}  

As already remarked, from our point of view the proof of Theorem \ref{sel-K-vanishing-teo} is interesting because it shows that, in the case of analytic rank zero, the vanishing of (almost all) the $p$-Selmer groups $\Sel_p(E/K)$ of an elliptic curve $E_{/\Q}$ can be obtained (at least when $K$ is an imaginary quadratic field) without resorting to auxiliary results in rank one as in the classical arguments due to Kolyvagin. 

\appendix

\section{Cohomology and Galois extensions} \label{appendix}

The purpose of this appendix is to describe certain Galois-theoretic properties of field extensions cut out by cohomology classes that are used in the main body of the paper. Although of an elementary nature, these results are somewhat hard to find in the literature, so for the convenience of the reader we decided to include them here.

Quite generally, in the following we let $F$ be a field (of arbitrary characteristic, though we apply our results only for
number fields) and let $F^s$ be the separable closure of $F$ in a fixed algebraic closure $\bar F$. All field extensions of $F$ that we consider will be contained in $F^s$, so we need not bother about separability issues (in particular, an extension $E/K$ is Galois if and only if it is normal).

For any extension $K/F$ contained in $F^s$ let $G_K:=\Gal(F^s/K)$. For any Galois extension $K/F$ contained in $F^s$ and any abelian discrete $G_K$-module (respectively, $\Gal(K/F)$-module) $M$ let $H^i(G_K,M)$ (respectively, $H^i(\Gal(K/F),M)$) denote the $i$-th continuous cohomology groups of $G_K$ (respectively, $\Gal(K/F)$) with values in $M$.

Fix a $G_F$-module $M$ and a Galois extension $K/F$. For any Galois extension $E/K$, the group $H^1(G_E,M)$ has a structure of a right $G_K$-module; the structure map is denoted $c\mapsto c^g$ for $c\in H^1(G_E,M)$ and $g\in G_K$, and defined by
\[ (c^g)(h):=g\bigl(c(g^{-1}hg)\bigr) \]
for all $h\in G_E$. Since the normal subgroup $G_E\subset G_K$ acts trivially on $H^1(G_E,M)$, this group becomes a $\Gal(E/K)$-module. Denote
\[ 0\longrightarrow H^1(\Gal(E/K),M^{G_E})\xrightarrow{{\rm inf}_{E/K}}H^1(G_K,M)\xrightarrow{\res_{E/K}}H^1(G_E,M)^{\Gal(E/K)} \]
the inflation-restriction exact sequence in Galois cohomology.

Fix a Galois extension $E$ of $F$ contained in $F^s$ and containing $K$ such that the subgroup $G_E$ of $G_F$ acts trivially
on $M$. In this case,
\[ H^1(G_E,M)^{\Gal(E/K)}=\Hom_{\Gal(E/K)}(G_E,M). \]
Take a non-zero class $s\in H^1(G_K,M)$ and assume, to avoid trivialities, that $\bar s:=\res_{E/K}(s)\neq 0$. This is the case, for example, if $M$ is irreducible as a $G_E$-module because in this situation $M^{G_E}=0$ and $\res_{E/K}$ is injective.

Let $E(s)$ denote the extension of $E$ cut out by $s$, that is, the field $E(s)$ such that the kernel of $\bar s$ is $G_{E(s)}$. Since $\bar s\neq 0$, $E(s)\neq E$. Moreover, the extension $E(s)/E$ is Galois. If $h\in G_E$ and $g\in G_F$ then $g^{-1}hg\in G_E$ because $E/F$ is Galois. In particular, if $h\in G_E$ and $g\in G_F$ then $g^{-1}hg\in G_{E(s)}$ if and only if $\bar s(g^{-1}hg)=0$.

\begin{prop} \label{app-prop-1}
The extension $E(s)/K$ is Galois.
\end{prop}

\begin{proof} The extension $E(s)/K$ is Galois if and only if $G_{E(s)}$ is a normal subgroup of $G_K$ or, equivalently, if and only if  $\bar s(g^{-1}hg)=0$ for all $h\in G_{E(s)}$ and $g\in G_K$. Since $\bar s\in\Hom_{\Gal(E/K)}(G_E,M)$, one has
\[ g\bigl(\bar s(g^{-1}hg)\bigr)=\bar s(h)=0 \]
for all $h\in G_{E(s)}$ and $g\in G_K$, and we are done. \end{proof}
The Galois group $\Gal(E(s)/E)$ is not trivial and injects into $M$ as a $G_K$-module, hence
\begin{equation} \label{N}
\Gal(E(s)/E)\simeq N(s)
\end{equation}
for a non-trivial $G_K$-submodule $N(s)$ of $M$.

\begin{prop} \label{app-prop-2}
The group $\Gal(E(s)/K)$ is isomorphic to the semidirect product
\[ \Gal(E(s)/K)\simeq N(s)\rtimes \Gal(E/K), \]
where the quotient $\Gal(E/K)$ acts on the abelian normal subgroup $N(s)$ by $n\mapsto \tilde g(n)$ for $n\in N(s)$ and
$g\in\Gal(E/K)$, where $\tilde g$ is any lift of $g$ to $G_K$.
\end{prop}

\begin{proof} The natural action by conjugation of $\Gal(E/K)$ on $\Gal(E(s)/E)$ translates into an action on $N(s)$ via the isomorphism \eqref{N}, and this gives $\Gal(E(s)/K)$ a structure of semidirect product as above. \end{proof}

In general, the extension $E(s)/F$ is not Galois. The canonical action of $\Gal(K/F)$ on $H^1(G_K,M)$ induces an action of
$\Gal(K/F)$ on $H^1(G_E,M)^{\Gal(E/K)}$ by restriction. Explicitly, let $c\in H^1(G_E,M)^{\Gal(E/K)}$ and $g\in\Gal(K/F)$, choose an extension $\tilde g\in G_F$ of $g$ and set
\[ (c^g)(h):=\tilde g\bigl(c(\tilde g^{-1}h\tilde g)\bigr). \]
Since $K/F$ is Galois, $G_K$ is normal in $G_F$ and the above action does not depend on the choice of $\tilde g$ because $c$ is $G_K$-invariant. With $s$ as before, for any $g\in\Gal(K/F)$ consider the extension $E(s^g)$ of $E$ cut out by $\bar s^g$. Moreover, denote $L$ the composite of the extensions $E(s^g)$ with $g$ ranging over $\Gal(K/F)$.

\begin{prop} \label{app-prop-3}
The extension $L/F$ is Galois.
\end{prop}

\begin{proof} The extension $L/F$ is Galois if and only if $G_L$ is a normal subgroup of $G_F$ or, equivalently, if and only if for any $h\in G_L$ one has $\bar s^g(k^{-1}hk)=0$ for all $g\in\Gal(K/F)$ and all $k\in G_F$. If $\tilde g\in G_F$ is an extension of $g$ then 
\[ (\bar s^g)(k^{-1}hk)=\tilde g\bigl(\bar s(\tilde g^{-1}k^{-1}hk\tilde g)\bigr). \] 
On the other hand, by assumption, if $g$ and $\tilde g$ are as above then
\[ k\tilde g\bigl(\bar s(\tilde g^{-1}k^{-1}hk\tilde g)\bigr)=\bar s^{\bar kg}(h)=0 \] 
with $\bar k:=k|_K$, hence
\[ \bar s\bigl(\tilde g^{-1}k^{-1}hk\tilde g\bigr)=0. \]
The proposition is proved. \end{proof}

Assume from now on that the $G_K$-module $M$ is \emph{finite} and \emph{irreducible}. Then $N(s^g)=M$ for all $g\in\Gal(K/F)$. Fix two distinct elements $g$ and $h$ in $\Gal(K/F)$. The group
\[ \Gal\bigl(E(s^g)/E(s^g)\cap E(s^h)\bigr) \] 
is identified via $s^g$ with a $G_K$-submodule of $M$, hence (by the irreducibility of $M$) either the extensions $E(s^g)$ and $E(s^h)$ coincide or they are linearly disjoint over $E$. In any case, $\Gal(L/E)$ is isomorphic to a product of copies of $M$ indexed by a subset $S\subset\Gal(K/F)$ which is minimal among the subsets of $\Gal(K/F)$ with the property that $\bigl\{E(s^g)\mid g\in S\bigr\}$ is equal to $\bigl\{E(s^g)\mid g\in\Gal(K/F)\bigr\}$. The isomorphism
\[ \Gal(L/E)=\prod_{g\in S}\Gal(E(s^g)/E)\simeq M^{\#S} \]
is explicitly given by $(h_g)_{g\in S}\mapsto\bigl(\bar s^g(h_g)\bigr)_{g\in S}$.

Let $h\in\Gal(E(s^g)/E)$ for some $g\in S$, and assume $h\neq1$. Then $\bar s^g(h)\neq 0$ and $\bar s^t(h)=0$ for all $t\in S$ with $t\neq g$. Let $k\in\Gal(E/F)$ and choose an extension $\tilde k\in G_F$ of $k$. Then
\begin{equation} \label{app-eq-1}
\bar s^{\bar kr}\bigl(\tilde kh\tilde k^{-1}\bigr)=\tilde k\bar s^r(h)
\end{equation}
for all $r\in\Gal(K/F)$, where $\bar k:=k|_K$, hence $\tilde kh\tilde k^{-1}\in\Gal\bigl(E(s^{\bar kg})/E\bigr)$. The last group is equal to $\Gal(E(s^u)/E)$ for a unique $u\in S$. Equality \eqref{app-eq-1} for $r=g$ shows that $k$ acts on $\Gal(L/E)$ as
\begin{equation} \label{app-eq-2}
\bar s^u\bigl(\tilde kh\tilde k^{-1}\bigr)=\bigl(\bar s^u(s^{\bar kg})^{-1}\bigr)\tilde k\bar s^g(h),
\end{equation} 
with $h\in\Gal(E(s^g)/E)$ and $g\in S$.

Equality \eqref{app-eq-2} shows that the action of $k\in\Gal(E/F)$ on an element $(\bar s^g(h_g))_{g\in S}\in M^{\#S}$ is essentially given by the action of $\tilde k$ on $\bar s^g(h_g)$, up to the action of the automorphism $\bar s^u(s^{\bar kg})^{-1}$ of $M$ which gives a permutation of the components.

\begin{prop} \label{app-prop-4}
The group $\Gal(L/K)$ is isomorphic to the semidirect product
\[ \Gal(L/K)\simeq M^{\#S}\rtimes\Gal(E/F), \]
where the action of the quotient $\Gal(E/F)$ on the abelian normal subgroup $M^{\#S}$ is explicitly described by \eqref{app-eq-2}.
\end{prop}

\begin{proof} Clear from the above discussion. \end{proof}

\end{document}